\documentclass[10pt,a4paper,openany]{article}
\usepackage[utf8]{inputenc}

\usepackage{appendix}
\usepackage{amscd,amsmath,amstext,amsfonts,amsbsy,amssymb,amsthm,eufrak}
\usepackage[unicode,colorlinks,linkcolor = blue,citecolor = red]{hyperref}
\usepackage{nccmath}
\usepackage{framed}
\usepackage{authblk}
\usepackage{hyperref,xcolor}
\usepackage{graphicx}
\usepackage{dsfont}
\usepackage[normalem]{ ulem }
\usepackage{soul, cancel}
\numberwithin{equation}{section}
\newtheorem{thm}{Theorem}[section]
\newtheorem{lem}[thm]{Lemma}
\newtheorem{proposition}[thm]{Proposition}
\newtheorem{corollary}[thm]{Corollary}

\newtheorem{definition}[thm]{Definition}
\newtheorem{thdef}[thm]{Theorem-Definition}
\newtheorem*{thm*}{Theorem}

\newtheorem{remark}{Remark}

\newcommand{\Var}{\mbox{Var}}

\newcommand{\sinc}{\mbox{sinc}}
\newcommand{\Tr}{\mbox{Tr}}
\def\Im{\textrm{Im}}
\def\Re{\textrm{Re}}

\def\N{\mathbb{N}}

\def\R{\mathbb{R}}
\renewcommand\C{\mathbb{C}}
\def\E{\mathbb{E}}

\def\ind{{\mathchoice {\rm 1\mskip-4mu l} {\rm 1\mskip-4mu l}
		{\rm 1\mskip-4.5mu l} {\rm 1\mskip-5mu l}}}

\newcommand\dd{\mathrm{d}}
\newcommand{\norm}[1]{\left\lVert#1\right\rVert}

\title{Statistical deconvolution of the free Fokker-Planck equation at fixed time}
\author{Mylène Ma\"ida\footnote{Univ. Lille, CNRS, UMR 8524 - Laboratoire Paul Painlevé, F-59000 Lille, France}, \qquad Tien Dat Nguyen\footnote{Laboratoire  de Math\'ematiques, UMR 8628, Universit\'e Paris Sud,
91405 Orsay Cedex France}, \qquad  Thanh Mai Pham Ngoc\footnote{Laboratoire  de Math\'ematiques, UMR 8628, Universit\'e Paris Sud,
91405 Orsay Cedex France},\\ Vincent Rivoirard\footnote{Ceremade, CNRS, UMR 7534, Universit\'e Paris-Dauphine, PSL Research University, 75016 Paris,
France},\qquad  Viet Chi Tran\footnote{LAMA, Univ Gustave Eiffel, UPEM, Univ Paris Est Creteil, CNRS, F-77447, Marne-la-Vall\'ee, France}}

\begin{document}

\maketitle

\begin{abstract}
We are interested in reconstructing the initial condition of a non-linear partial differential equation (PDE), namely the Fokker-Planck equation, from the observation of a Dyson Brownian motion at a given time $t>0$. The Fokker-Planck equation describes the evolution of electrostatic repulsive particle systems, and can be seen as the large particle limit of correctly renormalized Dyson Brownian motions. The solution of the Fokker-Planck equation can be written as the free convolution of the initial condition and the semi-circular distribution. We propose a nonparametric estimator for the initial condition obtained by performing the free deconvolution via the subordination functions method. This statistical estimator is original as it involves the resolution of a fixed point equation, and a classical deconvolution by a Cauchy distribution. This is due to the fact that, {in free probability,} the analogue of the Fourier transform  is the R-transform, related to the Cauchy transform. In past literature, there has been a focus on the estimation of the initial conditions of linear PDEs such as the heat equation, but to the best of our knowledge, this is the first time that the problem is tackled for a non-linear PDE. The convergence of the estimator is proved and the integrated mean square error is computed, providing rates of convergence similar to the ones known for non-parametric deconvolution methods. Finally, a simulation study illustrates the good performances of our estimator.
\end{abstract}

AMS 2000: 35Q62; 65M32; 62G05; 46L53; 35R30; 60B20; 46L54


\section{Introduction}

\subsection{Motivations}
Letting the initial condition of a partial differential equation (PDE) be random is interesting for considering complex phenomena or for introducing uncertainty and irregularity in the initial state. There is a large literature on the subject, and we can mention that this has been studied for the Navier-Stokes equation, to account for the turbulence arising in fluids with high velocities and low viscosities (see \cite{vishikfursikov,constantinwu}), for the Burgers equation that is used in astrophysics (see \cite{burgers,bertoingiraudisozaki,giraud2003,flandoli2018} or also the survey by \cite{vergassollaetal}), for the wave equations, to study the solutions with low-regularity initial data (see \cite{burqtzevtkovI,burqtzevtkovII,tzvetkov}) or for the Schr\"odinger PDE (see \cite{bourgain}). The Burgers PDE or the vortex equation, associated to the Navier-Stokes PDE by considering the curl of the velocity, are of the McKean-Vlasov type as introduced and studied in \cite{sznitman,meleardcime}. Numerical approximations of such PDEs with random initial conditions have been considered in \cite{talayvaillant,transolstat}. In this paper, we are interested in the Fokker-Planck PDE which is another case of McKean-Vlasov PDE \cite{carrillomccannvillani}. This equation models the motion of particles with electrostatic repulsion and a probabilistic interpretation that we will adopt has been considered in \cite{bianespeicher}.\\

A question naturally raised in this context is to estimate the random initial condition, given the observation of the PDE solution at a given fixed time $t>0$. For linear PDEs, this inverse problem is solved by deconvolution techniques, and this has been explored for PDEs such as the heat equation or the wave equation by Pensky and Sapatinas \cite{penskysapatinas09,penskysapatinas10}. For the 1d-heat equation, it is known that the solution at time $t$, say $\nu_t(dx)$, is the convolution of the initial condition $\nu_0(dx)$ with Green function $G_t$, which is a Gaussian transition function associated with the standard Brownian motion $(B_t)_{t\geq 0}$. The probabilistic interpretation of the heat equation is built on this observation, and $\nu_t$ can be viewed as the distribution of $X_t=X_0+B_t$ where $X_0$ is distributed as $\nu_0$. Taking the Fourier transforms changes the convolution problem into a multiplication, which paves the way to reconstruct the initial condition.\\
Here, we are interested in estimating the initial condition of a non-linear PDE, namely the Fokker-Planck equation, from the observation of its solution at time $t$. Recall that the Fokker-Planck equation is:
\begin{equation}\label{eq:fokkerplanck}
\partial_t p(t,x)=-\partial_x \int_{\R^2} Hp(t,x) p(t,x)\dd x,
\end{equation}with
\[Hp(t,x)= \lim_{\varepsilon \rightarrow 0} \int_{\R \setminus [x-\varepsilon,x+\varepsilon]} \frac{1}{x-y}p(t,y)\dd y, \]
and for $t\in\R_+$, $x\in \R$, and initial condition $p_0(x) \in L^1(\R)$. Contrarily to the examples considered in \cite{penskysapatinas09,penskysapatinas10}, this PDE is non-linear of the McKean-Vlasov type with logarithmic interactions. To the best of our knowledge, this is the first work devoted to the deconvolution of a non-linear PDE to recover the initial condition. The choice of this equation is motivated by its strong similarities with the heat equation: the standard Brownian motion of the probabilistic interpretation is replaced here by the free Brownian motion $(\mathbf{h}_t)_{t\geq 0}$ (operator-valued), and the usual convolution by a Gaussian distribution is replaced by the free convolution by a semi-circular distribution $\sigma_t$ characterized by its density with respect to the Lebesgue measure:
\begin{equation}\label{def:semicircle}
\sigma_t(dx)=\frac{1}{2\pi t}\sqrt{4t-x^2} \ind_{[-2\sqrt{t},2\sqrt{t}]}(x) \ dx.
\end{equation}
If $\mathbf{x}_0$ admits the spectral measure $\mu_0$, then $\mathbf{x}_t=\mathbf{x}_0+\mathbf{h}_t$ admits
\begin{equation}\label{eq:freeconvolution}
\mu_t= \mu_0 \boxplus \sigma_t,
\end{equation}as spectral measure, where the operation $\boxplus$ is the free convolution and has been introduced by Voiculescu in \cite{voiculescu}. It can be proved that {the density $p(t,\dot )$ of} $\mu_t$ solves \eqref{eq:fokkerplanck}.

For the Fokker-Planck equation, the inverse problem boils down to a free deconvolution, where it was a usual deconvolution for the heat equation. Recently, the problem of free deconvolution has been studied by Arizmendi, Tarrago and Vargas \cite{Tarrago1}. To solve \eqref{eq:freeconvolution} in a general setting, subordination functions are used. Here, if the Cauchy transform of a measure $\mu$ is defined as $G_{\mu}(z)=\int_{\R} (z-x)^{-1}\dd\mu(x)$ for $z\in \C^+$, where $\C^+$ is the set of complex numbers with positive imaginary part, the subordination function $w_{fp}(z)$ at time $t$ is related to $G_{\mu_t}$ by the functional equation
\begin{equation}\label{eq:intro}
w_{fp}(z)=z+t G_{\mu_t}(w_{fp}(z)).
\end{equation}From this, we can recover $G_{\mu_0}$ with the formula $G_{\mu_0}(z)=G_{\mu_t}(w_{fp}(z))$ and thus $p_0$ (see Lemma \ref{lem:Gmu0} and \eqref{eq:lien_f_G} in the paper). More precisely, we prove in Section~\ref{sec:defest} that for any $\gamma>2\sqrt{t}$, $f_{\mu_0 * \mathcal C_{\gamma}}$ the density of the classical convolution of $\mu_0$ with the Cauchy distribution of parameter $\gamma$, defined by its density
\[f_\gamma(x):= \frac{\gamma}{\pi(x^2+\gamma^2)},\]
satisfies
\begin{equation}\label{deconv-intro}
f_{\mu_0 * \mathcal C_{\gamma}}(x)=  \frac{1}{\pi t} \left[ \gamma  - \Im w_{fp}(x+ i \gamma)\right], \quad x \in \R.
\end{equation}
Then, estimating $p_0$, the density of $\mu_0$, requires an estimation of the subordination function $w_{fp}$ combined with a classical deconvolution step from a Cauchy distribution.

\subsection{Observations}
Additionally to the free deconvolution problem, our observation does not consist in the operator-valued random variable $\mathbf{x}_t$ but in its matricial counterpart. More precisely, we observe a matrix $X^n(t)$ for a given $t>0$, assumed to be fixed in the sequel, where
\begin{equation}\label{eq:Xn}
X^n(t)=X^n(0)+H^n(t),\qquad  t\geq 0
\end{equation}with $X^n(0)$  a diagonal matrix whose entries are the ordered statistic $\lambda^n_1(0)<\cdots <\lambda_n^n(0)$ of a vector $(d_i^n)_{i\in \{1,\dots n\}}$ of $n$ independent and identically distributed (i.i.d.) random variables distributed as $\mu_0(dx)=p_0(x)\ dx$, absolutely continuous with respect to the Lebesgue measure on $\R$, and  $H^n(t)$  a standard Hermitian Brownian motion, as defined in Definition \ref{def:Hermit.BrowMotion}. The purpose is to estimate $p_0$ without observing directly
the initial condition $X^n(0)$. As the distribution of $H^n(t)$ is invariant by conjugation, choosing $X^n(0)$ to be a diagonal matrix is not restrictive. It is known, see \cite[page 249, section 4.3.1]{AGZ}, that the eigenvalues $(\lambda_1^n(t),\cdots,\lambda_n^n(t))$ of $X^n(t)$ solve the following {system of} stochastic differential equations (SDE):
\begin{equation}\label{eq:SDE.lambda}
\dd \lambda^{n}_i(t)= \frac{1}{\sqrt{n}}\dd \beta_i(t)+ \frac{1}{n}\sum_{j\not= i}\frac{\dd t}{\lambda^{n}_i(t)-\lambda^{n}_j(t)},\quad 1\leq i\leq n,
\end{equation}where $\beta_i$ are i.i.d. standard {real} Brownian motions. If we denote by
\begin{equation} \label{def:empiricalmeasure}
\mu^n_t=\frac{1}{n}\sum_{i=1}^n \delta_{\lambda_i^n(t)}
\end{equation}
the empirical measure of these eigenvalues at time $t$, then the process $(\mu^n_t)_{t \ge 0}$ converges weakly almost surely as $n$ goes to infinity to the process
{$(\mu_t)_{t \ge 0}$} with density $(p(t,\cdot))_{t \ge 0}$ solution of \eqref{eq:fokkerplanck}. For $n=1$, we recover the classical heat equation as the Dyson Brownian motion boils down to a standard Brownian motion. \\

\subsection{Contributions}
Relying on the analysis of Arizmendi et al. \cite{Tarrago1}, we provide, in Theorem-Definition \ref{th:defestimator}, a statistical estimator $\widehat{w}^n_{fp}(z)$ for the subordination function. As the Cauchy transform $G_{\mu_t}$ in \eqref{eq:intro} is not invertible on the whole domain $\C^+$, the subordination function $w_{fp}(z)$ will be defined only for $z\in \C_{2\sqrt{t}}$ where $\C_\gamma:=\{z\in \C^+,\ \Im(z)>\gamma\}$. We shall prove the following result.
\begin{proposition}\label{propsition.consistency.estimate.w_fp}Let $\gamma>2\sqrt{t}$. Suppose $\lambda^{n}(0)$ satisfies the condition
	\begin{equation}\label{C0}
	 \sup_{n \geq 1} \dfrac{1}{n} \sum_{i=1}^{n} \log \left( \lambda^{n}_{i}(0)^2 + 1 \right)<\infty \textrm{ almost surely (a.s.)}  
	\end{equation}
Then, we have:\\
(i) For any $z \in \mathbb{C}_{2\sqrt{t} }$, the estimator $\widehat{w}^{n}_{fp}(z)$ converges almost surely to $w_{fp}(z)$ as $n \rightarrow \infty$.\\
(ii) The convergence is uniform on $\C_\gamma$. \\
(iii) We have the following convergence rate on $\C_\gamma$:
\[ \sup_{n \in \N} \sup_{z \in \C_{\gamma}}  \mathbb{E} \left[ n\big| \widehat{w}^{n}_{fp}(z) - w_{fp}(z)  \big|^{2} \right]  <+\infty. \]
\end{proposition}
To obtain uniform convergence and fluctuations ((ii) and (iii)), we will need to restrict to strict subdomains of $\C_{2\sqrt{t}}$. The fluctuations (iii) are established in the line of the work of Dallaporta and F\'evrier \cite{Fevrier1}.\\

Proposition \ref{propsition.consistency.estimate.w_fp} is the crucial tool to reach the main goal of this paper, namely providing an estimator of $p_0$. As explained previously, we estimate $p_0$ by combining a free deconvolution step via the use of $\widehat{w}^{n}_{fp}$ with then a classical deconvolution step.
We define our final estimator $\widehat{p}_{0,h}$ via its Fourier transform, denoted $\widehat{p}_{0,h}^{\star}$; from Equation~\eqref{deconv-intro}, it is natural to define it as follows:
$$
\widehat{p}_{0,h}^{\star}(\xi) = e^{\gamma  |\xi| }. K_{h}^{\star}(\xi). \dfrac{1}{\pi t} \left[ \gamma -  \Im \hspace{0.05cm} \widehat{w}^n_{fp}(\cdot+ i \ \gamma )^{\star}(\xi) \right],\quad \xi\in\R.
$$
Note that, as usual in nonparametric statistics, the last expression depends on $K_{h}^{\star}$, a regularization term defined through the Fourier transform of a kernel function $K_{h}$ depending on a bandwidth parameter $h$. See Equation \eqref{Fourier.f_mu0.hat} in Definition~\ref{def:est} for more details. \\

We study theoretical properties of $\widehat{p}_{0,h}$ by deriving asymptotic rates of the mean integrated square error of $\widehat{p}_{0,h}$ decomposed as the sum of bias and variance terms. The study of the variance term is intricate and is based on  the sharp controls of the difference $\widehat{w}^{n}_{fp}(z) - w_{fp}(z)$ provided by Proposition~\ref{propsition.consistency.estimate.w_fp}. We show in Theorem~\ref{variance} that the variance term is of order $\frac{e^{\frac{2\gamma}{h}}}{n}$ as desired for deconvolution with the Cauchy distribution with parameter~$\gamma$. The bias term is driven by the smoothness properties of the function $p_0$. In particular, when we assume that $p_0$ belongs to a space of supersmooth densities (see \eqref{def:Sobolev}), we can establish convergence rates, after an appropriate (non-adaptive) choice of the bandwidth parameter $h$. For instance, if $$\int_{\R}  |p_0^\star(\xi)|^2 e^{2a |\xi|}\dd \xi \leq L$$ for $0<L<\infty$, then
\[\E\Big[\|\widehat{p}_{0,h}-p_0\|^2\big] = O\big(n^{-\frac{a}{a+\gamma}}\big).\]
The previous rate is optimal when we address the statistical deconvolution problem involving the Cauchy distribution with parameter $\gamma$. See Corollary~\ref{thm:MISE} for more details and more general results that establish the optimality of our procedure.
 Note that the exponent in the previous bound reflects the difficulty of our statistical problem: the larger $\gamma$, the smaller the rate. Remembering that $\gamma$ is connected to the observational time $t$ through the condition $\gamma>2\sqrt{t}$, it means that for the previous example, our estimate can achieve the polynomial rate $n^{-\frac{a}{a+2\sqrt{t}+\epsilon}}$ for any $\epsilon>0$.  The question of whether it is possible to consider smaller values for $\gamma$ constitutes a challenging problem. Adaptive choices for $h$ are also a very interesting issue. These problems will be investigated in another work.

\subsection{Overview of the paper}
In Section \ref{section:freedeconv}, we study the free deconvolution and explain the construction of the estimator $\widehat{p}_{0,h}$ of $p_0$. Existence results and properties of the subordination functions are precisely stated and proved. Then, in Section \ref{section:estimationw}, we prove Proposition \ref{propsition.consistency.estimate.w_fp}. In Section~\ref{section:MISE}, rates of convergence of $\widehat{p}_{0,h}$ are established. Numerical simulations are provided in Section~\ref{section:numerical}.

\paragraph*{Notations:} For any $z=u+iv\in\C+$, we denote $\sqrt{z}:=a+ib\in\C$ with
\[a=\sqrt{\frac{\sqrt{u^2+v^2}+u}{2}},\quad b=\sqrt{\frac{\sqrt{u^2+v^2}-u}{2}}.\]
We denote the Fourier transform of a function $g \in {\mathbb L}^{1}(\R)$ by $g^{\star}$:
\begin{equation}\label{Fourier}
g^\star(\xi):=\int_\R g(x)e^{ix\xi}dx,\quad \xi\in\R.
\end{equation}
\section{Free deconvolution of the Fokker-Planck equation}\label{section:freedeconv}
\subsection{Dyson Brownian motions}

Let us denote by $\mathcal{H}_{n}(\C)$ the space of $n$-dimensional matrices $H_n$ such that $\left(H_{n}\right)^{*} = H_{n}$.
\begin{definition}\label{def:Hermit.BrowMotion}
		Let $\big( B_{i,j}, \tilde{B}_{i,j}, 1 \leq i \leq j \leq n \big)$ be a collection of i.i.d. real valued standard Brownian motions, the  Hermitian Brownian motion, denoted $H^{n} \in \mathcal{H}_{n}(\C)$, is the random process with entries $\left\{ (H^{n}(t))_{k,l}, t \geq 0,1 \le k, l \leq n\right\}$ equal to
		\begin{equation}\label{eq:Hermit.BrowMotion}
		(H^{n})_{k,l} =
		\begin{cases}
		\dfrac{1}{\sqrt{2n}} \left( B_{k,l} + i\  \tilde{B}_{k,l} \right), &\textrm{if } k < l \\[0.5cm]
		\dfrac{1}{\sqrt{n}} B_{k,k}, &\textrm{if } k = l
		\end{cases}
		\end{equation}
\end{definition}

Let us now define the initial condition, that we will choose independent of the Hermitian Brownian motion $H^n$. Recall that $\mu_0$ is a probability measure with density $p_0(x)$ with respect to the Lebesgue measure on $\R$. Without loss of generality, we can choose the initial condition
$X^n(0)$ to be a diagonal matrix, with entries $(\lambda^n_1(0), \ldots, \lambda^n_n(0))$ the ordered statistics of i.i.d. random variables {$(d_i^n)_{1\le i \le n}$} with distribution $\mu_0$.

For $t \geq 0$, let $\lambda^{n}(t) = \big( \lambda^{n}_{1}(t), \dots, \lambda^{n}_{n}(t) \big)$ denote the ordered collection of eigenvalues of
\begin{equation}\label{eq:Dyson.BrowMotion}
X^{n}(t) = X^{n}(0) + H^{n}(t).
\end{equation}

\begin{thm}[Dyson]\label{thm.Dyson}
The process $\big( \lambda^{n}(t) \big)_{t \geq 0}$ is the unique solution in $C \left( \mathbb{R}_{+}, \mathbb{R}^{n} \right)$ of
the system
	\begin{equation}
	d\lambda^{n}_{i}(t) = \dfrac{1}{\sqrt{n}}d\beta_{i}(t) + \dfrac{1}{n} \sum_{j\not= i} \dfrac{dt}{\lambda^{n}_{i}(t) - \lambda^{n}_{j}(t)}, \quad 1 \leq i \leq n, \tag{\ref{eq:SDE.lambda}}
	\end{equation}
with initial condition $\lambda^{n}_{i}(0)$ and where $\beta_{i}$ are i.i.d. real valued standard Brownian motions.
With probability one and for all $t> 0$, $ \lambda^{n}_{1} (t)<\ldots <\lambda^{n}_{n} (t) $.
\end{thm}


Moreover, if for
fixed $T > 0$, we denote by $\mathcal{C} \big(\left[0,T\right], \mathcal{M}_{1}\left(\mathbb{R}\right) \big)$ the space of continuous processes from $\left[0 , T\right]$ into $\mathcal{M}_{1}\left(\mathbb{R}\right),$ the space of probability measure on $\mathbb{R}$, equipped with its weak topology, we now prove the convergence of the process of empirical measures $\mu^{n}$ as defined in \eqref{def:empiricalmeasure}, viewed as an element of $C \big(\left[0,T\right], \mathcal{M}_{1}\left(\mathbb{R}\right) \big)$.

\begin{proposition}\label{prop:LGN.mu.t}
Under Assumption \eqref{C0}, for any fixed time $T < \infty$, $\big( \mu^{n}_{t} \big)_{t \in [0,T]}$ converges almost surely in $\mathcal{C} \big( [0,T] , \mathcal{M}_{1}\left(\mathbb{R}\right) \big)$. Moreover, its limit is the unique measure-valued process $\left( \mu_{t} \right)_{t \in [0,T]}$
whose densities satisfy \eqref{eq:fokkerplanck} with initial condition {$p_0.$}
\end{proposition}

For deterministic initial conditions, Theorem \ref{thm.Dyson} and Proposition \ref{prop:LGN.mu.t} are classical results and we refer to \cite[Section 4.3]{AGZ} for a proof. Both results can be easily extended to random initial conditions, independent of the Hermitian Brownian motion itself. For details, we refer to \cite{datthesis}.

\subsection{Free deconvolution by subordination method}

Our starting point is \eqref{eq:freeconvolution}, for a fixed time $t >0.$ Recovering $\mu_0$ knowing $\mu_t$ is a free deconvolution problem. The generic problem of free deconvolution has been introduced and studied by Arizmendi et al. \cite{Tarrago1} with the use of the Cauchy transform instead of the Fourier transform. Before stating their result, we need to introduce a few notations and definitions.

\begin{definition}\label{def:Cauchy.transform}
Let $\mu$ be a probability measure on $\R$, the Cauchy transform of $\mu$ is defined by:
\begin{equation}
G_\mu(z)=\int_\R \frac{d\mu(x)}{z-x}, \quad z \in \mathbb{C} \setminus \mathbb{R}.
\end{equation}
 \end{definition}
The fact is that $G_{\mu}\left( \overline{z} \right) = \overline{G_{\mu}(z)}$, so the behavior of the Cauchy transform in the lower half-plan $\mathbb{C}^{-} = \left\{ z \in \C | \Im(z) < 0 \right\}$ can be determined by its behavior in the upper half-plan $\mathbb{C}^{+} = \left\{ z \in \C | \Im(z) > 0 \right\}$.
The function $G_\mu$ is a bijection from a neighbourhood of infinity to a neighbourhood of zero (see \cite{Biane1} for example) and we can define the $R$-transform of $\mu$ by:
\[
R_\mu(z)=G_\mu^{<-1>}(z)-\frac{1}{z},
\]where $G_\mu^{<-1>}(z)$ is the inverse function of $G_\mu$ on a proper neighbourhood of zero. This $R$-transform plays the role of the logarithm of the Fourier transform for the free convolution in the sense that for any probability measures $\mu_1$ and $\mu_2$,
\begin{equation}\label{eq:R-transform}
R_{\mu_1\boxplus \mu_2}=R_{\mu_1}+R_{\mu_2}.
\end{equation}
Using this formula for statistical deconvolution requires the computation of two inverse functions, and Arizmendi et al.  \cite{Tarrago1} propose to use subordination functions {which also characterize the free convolution as in} \eqref{eq:R-transform}. \\

Let us recall the definition of subordination functions due to Voiculescu \cite{voiculescu93}. We first introduce $F_\mu(z)=1/G_\mu(z)$.
As $G_\mu$ does not vanish on $\C^+$, $F_\mu$ is well defined on $\C^+$. Then:
\begin{thdef}\label{def:subordinationfunction}There exist unique subordination functions $\alpha_{1}$ and $\alpha_{2}$ from $\C^+$ onto $\C^+$ such that:\\
(i) for $z \in \C^+$, $\Im\big(\alpha_{1}(z)\big) \geq \Im(z)$ and $\Im\big(\alpha_{2}(z)\big) \geq  \Im(z)$, with
		\[\lim_{y\rightarrow +\infty} \frac{\alpha_{1}(iy)}{iy}=\lim_{y\rightarrow +\infty} \frac{\alpha_{2}(iy)}{iy}=1.\]
(ii) for $z \in \C^+$, $F_{\mu_1\boxplus \mu_2}(z)=F_{\mu_1}(\alpha_{1}(z))=F_{\mu_2}(\alpha_{2}(z))$
		and $\alpha_1(z)+\alpha_2(z)=F_{\mu_1\boxplus \mu_2}(z)+z$.
\end{thdef}

Using this result, Belinschi and Bercovici \cite[Theorem 3.2]{belinschibercovici} introduce a fixed-point construction of the subordination functions, which  Arizmendi et al.  \cite{Tarrago1} adapt for the deconvolution problem. We state their result in the special case of the deconvolution by a semi-circular distribution defined in \eqref{def:semicircle}. In this case, we have an explicit formula for its Cauchy transform $G_{\sigma_t}(z)$ and its reciprocal function $F_{\sigma_t}(z):$
\begin{equation}\label{eq.relation.FGSigmat}
G_{\sigma_t}(z)=\frac{z-\sqrt{z^2-4t}}{2t},\qquad \mbox{ and  }\qquad
z - F_{\sigma_t}(z) = t \ G_{\sigma_t}(z).
\end{equation}
Before stating the result, let us define, for any $\gamma > 0,$
\[ \mathbb{C}_{\gamma} = \left\{ z \in \mathbb{C}^{+}  \big| \textrm{Im}(z) > \gamma \right\}. \]
These domains will appear since $G_\mu$ is not invertible on the whole plane $\mathbb{C}$.


\begin{thm}\label{thm.Fixpoint}
There exist unique subordination functions $w_{1}$ and $w_{fp}$ from $\C_{2\sqrt{t}}$ onto $\C^+$ such that following properties are satisfied.\\
(i) For $z \in \C_{ 2\sqrt{t} }$, $\Im\big(w_{1}(z)\big) \geq \dfrac{1}{2} \Im(z)$ and $\Im\big(w_{fp}(z)\big) \geq \dfrac{1}{2} \Im(z)$, and also
		\[\lim_{y\rightarrow +\infty} \frac{w_{1}(iy)}{iy}=\lim_{y\rightarrow +\infty} \frac{w_{fp}(iy)}{iy}=1.\]
(ii) For $z \in \C_{2\sqrt{t}}$ :
\begin{equation}\label{eq.Subordi.ii} F_{\mu_0}(z)=F_{\sigma_t}(w_{1}(z))=F_{\mu_t}(w_{fp}(z)).\end{equation}
(iii) For $z\in \C_{2\sqrt{t}}$ :
		\begin{equation}
		w_{fp}(z)=z+w_{1}(z)-F_{\mu_0}(z). \label{eq.Subordi.func}
		\end{equation}		
(iv) Denote $h_{\sigma_t}(w) = w - F_{\sigma_t}(w) = t\ G_{\sigma_t}(w)$ and $\widetilde{h}_{\mu_t}(w)=w+F_{\mu_t}(w)$ on $\mathbb{C}^{+}$. We can define the function ${L}_z$ as
		\begin{align}
		 L_z(w) :&= h_{\sigma_t} \big( \widetilde{h}_{\mu_t}(w) - z \big) + z  \nonumber\\
		&= t. G_{\sigma_t} \big(  \widetilde{h}_{\mu_t}(w) - z \big) + z . \label{fixpoint.func.form}
		\end{align}
For any $z\in \C_{2\sqrt{t}}$, we have
\begin{equation}\label{eq:fix-point}L_z\big(w_{fp}(z)\big)=w_{fp}(z),
\end{equation}
and for all $w$ such that $\Im(w) > \frac{1}{2} \Im(z)$, the iterated function $L_z^{\circ m}(w)$ converges to $w_{fp}(z) \in \mathbb{C}^{+}$ when $m\rightarrow +\infty$.
\end{thm}

One difference between Theorem \ref{thm.Fixpoint} and Theorem-definition \ref{def:subordinationfunction} lies in the fact that the subordination functions are expressed in terms of $F_{\mu_0\boxplus \sigma_t}$ and $F_{\sigma_t}$ whereas in Theorem-definition \ref{def:subordinationfunction} it would have been $F_{\mu_0}$ and $F_{\sigma_t}$. Here the restriction to the domain $\C_{2\sqrt{t}}$ comes from the fact that $\Im(\widetilde{h}_{\mu_t}(w) - z )$ appearing in the definition \eqref{fixpoint.func.form} of $L_z$ has to be positive.\\

The proof of Theorem \ref{thm.Fixpoint} is postponed to the last subsection of this section, Section \ref{sec:proofTheoremFixPoint}.
We now explain how the subordination functions allow us to construct the estimator of $p_0$.

\subsection{Construction of the estimator of $p_0$}\label{sec:defest}

\paragraph{Overview of the estimation strategy} Based on Theorem \ref{thm.Fixpoint}, we devise the estimation strategy of the paper. The theorem allows us to get the subordination function $w_{fp}$ as a fixed point of $L_z$. From there, we will be able to recover the Cauchy transform of the initial condition $\mu_0$ from $w_{fp}$, as stated in the following lemma proved at the end of the section:
\begin{lem}\label{lem:Gmu0}
For any $z \in \C_{ 2\sqrt{t} },$
\begin{equation}\label{eq:lemGmu0} G_{\mu_0}(z) = \frac{1}{t}(w_{fp}(z)-z)=G_{\mu_t}\big(w_{fp}(z)\big). \end{equation}
Consequently,
	\begin{equation}\label{estimee:wfp}
	\left| w_{fp}(z) - z \right| \leq \sqrt{t}.
	\end{equation}
\end{lem}
Moreover, for any $\gamma >0,$ if we denote by $\mathcal C_\gamma$ the centered Cauchy distribution with density:
\[f_\gamma(x):= \frac{\gamma}{\pi(x^2+\gamma^2)},\]
one can check that, for any probability measure $\mu$ on $\R,$ the density $f_{\mu * \mathcal C_\gamma}$ of the classical convolution of $\mu$
by  $\mathcal C_\gamma$ is given, for $x \in \R,$ by
\begin{equation}\label{eq:lien_f_G} f_{\mu * \mathcal C_\gamma}(x)= - \frac{1}{\pi}\Im G_{\mu}(x+ i \gamma). \end{equation}
Using the expression of $G_{\mu_0}$ given by Lemma \ref{lem:Gmu0} with $\gamma>2\sqrt{t}$, we get that for any $x \in \R,$
\begin{equation}\label{eq:densiteconvolee}
f_{\mu_0 * \mathcal C_{\gamma}}(x)=  \frac{1}{\pi t} \left[ \gamma  - \Im w_{fp}(x+ i \gamma)\right].
\end{equation}
From this, we can recover the density $p_0$ of $\mu_0$ by a classical deconvolution of \eqref{eq:densiteconvolee} by $f_{\gamma}$. The subordination function $w_{fp}$ in \eqref{eq:densiteconvolee} is estimated using the second equality of Lemma \ref{lem:Gmu0}. In parallel with our work, Tarrago \cite{tarrago} has used the formula \eqref{eq:densiteconvolee} to perform spectral deconvolution in a more general setting (including the multiplicative free convolution), but neither the approximation of $w_{fp}$ by its estimator $\widehat{w}^n_{fp}$ defined Theorem-Definition \ref{th:defestimator}  below nor the (classical) deconvolution of the Cauchy distribution are treated, which are key difficulties encountered in our paper. Tarrago uses a different approach based on concentration inequalities when we use fluctuations in view of the work of F\'evrier and Dallaporta \cite{Fevrier1}. To prove the rates announced in the introduction, we need to establish very precise estimates of the error terms (see Section~\ref{section:MISE}).

\begin{proof}[Proof of Lemma \ref{lem:Gmu0}]
Now, from (\ref{eq.Subordi.func}),  \eqref{eq.Subordi.ii} and \eqref{eq.relation.FGSigmat}, we write for $z\in \C_{2\sqrt{t}}$,
\begin{equation*}
w_{fp}(z) = z + w_{1}(z) - F_{\mu_0}(z) = z + w_{1}(z) - F_{\sigma_t}\big( w_{1}(z) \big) =  z + t.  G_{\sigma_t} \big( w_{1}(z) \big).
\end{equation*}
So, we obtain
\begin{equation*} G_{\sigma_t}\big( w_{1}(z) \big) = \dfrac{1}{t} \big( w_{fp}(z) - z \big).
\end{equation*}Using again \eqref{eq.Subordi.ii}, we obtain both equalities of \eqref{eq:lemGmu0}. From there
	\begin{equation*}
	\left| w_{fp}(z) - z \right| = t.\left| G_{\sigma_t}(w_{1}(z))\right| = t.\left| G_{\mu_{t}}(w_{fp}(z))  \right|  \leq \dfrac{t}{| \Im(w_{fp}(z)) |}  \leq  \sqrt{ t},
	\end{equation*}using Theorem \ref{thm.Fixpoint} (i).
\end{proof}

\paragraph{Estimator of $p_0$}

We do not observe directly the measure $\mu_t$. The observation is the matrix $X^n(t)$ at time $t>0$ for a given $n$. From this observation, we can construct the empirical spectral measure as defined in \eqref{def:empiricalmeasure}. Then, for $z \in \C^{+}$, a natural estimator of $G_{\mu_t}(z)$ is obtained as follows:
\begin{equation}\label{def:Ghat}
\widehat{G}_{\mu^{n}_t}(z):=\int_{\R} \frac{d\mu^n_t(\lambda)}{z-\lambda}=\frac{1}{n}\sum_{j=1}^n \frac{1}{z-\lambda^{n}_j(t)}=\frac{1}{n}\mbox{tr}\Big(\big(zI_n-X_n(t) \big)^{-1}\Big).
\end{equation}

\begin{thdef}\label{th:defestimator}
There exists a unique fixed-point to the following functional equation in $w(z)$:
\begin{equation} \label{eq:fixed-point.empirical}
\dfrac{1}{t} (w(z)-z)= \widehat G_{\mu^{n}_t}(w(z)), \quad \mbox{ for } z \in \C_{ 2\sqrt{t} }
\end{equation}
This fixed-point is denoted by $\widehat{w}^n_{fp}(z)$. We have $\Im(\widehat{w}^{n}_{fp}(z))>\Im(z)/2$ and $\left| \widehat{w}^{n}_{fp}(z) - z \right| \leq \sqrt{t}$.
\end{thdef}

The theorem is proved at the end of this section. We shall prove in Section \ref{section:estimationw} that $\widehat{w}^n_{fp}(z)$ is a convergent estimator of $w_{fp}(z)$ and establish a fluctuation result associated with this convergence. This is the result announced in Proposition \ref{propsition.consistency.estimate.w_fp}. Let us now explain how the estimator of $p_0$ can be obtained from $\widehat{w}^n_{fp}(z)$.\\

Recall that the Fourier transform of the Cauchy distribution $\mathcal{C}_{\alpha}$ with $\alpha>0$ is $f^{\star}_{\alpha} (\xi) = e^{-\alpha \left|\xi\right|}$ for $\xi \in \R$. Performing the deconvolution from \eqref{eq:densiteconvolee}, the Fourier transform of $p_0$ is the division of the Fourier transform of the right-hand side of \eqref{eq:densiteconvolee} by $f^{\star}_{\gamma} (\xi)$ with $\gamma>2\sqrt{t}$. It is now classical to define our ultimate estimator for the density function $p_0$ from its Fourier transform:

\begin{definition}\label{def:est}
Let us consider a bandwidth $h>0$ and a regularizing kernel $K$. We assume that the kernel $K$ is such that its Fourier transform $K^\star$ is bounded by a positive constant $C_K<+\infty$ and has a compact support, say $[-1,1]$. We define the estimator $\widehat{p}_{0,h}$ of $p_0$  by its Fourier transform:
\begin{equation} \label{Fourier.f_mu0.hat}
\widehat{p}_{0,h}^{\star}(\xi) = e^{\gamma  |\xi| }. K_{h}^{\star}(\xi). \dfrac{1}{\pi t} \left[ \gamma -  \Im \hspace{0.05cm} \widehat{w}^n_{fp}(\cdot+ i \ \gamma )^{\star}(\xi) \right],
\end{equation}
where we have defined $K_h(\cdot)=\frac{1}{h}K\big(\frac{\cdot}{h}\big)$.
\end{definition}

Note that the assumption on $K$ ensures the finiteness of the estimator. These assumptions are for instance satisfied for $K(x)=\sinc(x)=\sin(x)/(\pi x)$ whose Fourier transform is $K^{\star}(\xi)=\mathds{1}_{[-1, 1]}(\xi)$, and in this case $C_K=1$. 
\subsection{Proof of Theorem \ref{thm.Fixpoint}}\label{sec:proofTheoremFixPoint}
The constants of Theorem \ref{thm.Fixpoint} are better than the ones of Arizmendi et al. \cite{Tarrago1} who work in  full generality. We sketch here the main steps of the proof in our context, using the explicit formula for the semi-circular distribution.\\

\noindent In the whole proof, we consider $z \in \mathbb{C}_{2\sqrt{t}}$.\\

\noindent \textbf{Step 1:} We first prove that the function ${L}_z(w) = h_{\sigma_t}\big( \widetilde{h}_{\mu_t}(w) - z \big) + z$ is well-defined and analytic on $\mathbb{C}_{\frac{1}{2} \Im(z)}$. Since $h_{\sigma_t}$ is defined on $\mathbb{C}^{+}$, we need to check that $\widetilde{h}_{\mu_t}(w) - z \in \mathbb{C}^{+}$ for $w \in \mathbb{C}_{\frac{1}{2} \Im(z)}$. This is satisfied since for such $w$,
		\begin{align}\label{etape2}
		\Im\big( \widetilde{h}_{\mu_t}(w) - z \big) = \Im \big( w + F_{\mu_t}(w) - z \big) \geq 2 \Im(w) - \Im(z) > 0,
		\end{align}
		where we have used $\Im F_{\mu_t}(w) \geq \Im(w)$ for the first inequality.
Indeed, if $w=w_1+iw_2$, we have
\[(F_{\mu_t}(w))^{-1}=G_{\mu_t}(w)=\int\frac{d\mu_t(x)}{w_1+iw_2-x}=\int \frac{(w_1-x)d\mu_t(x)}{(w_1-x)^2+w_2^2}-iw_2\int \frac{d\mu_t(x)}{(w_1-x)^2+w_2^2}\] and
\begin{align}\label{mino}
\Im(F_{\mu_t}(w))&=w_2\times\frac{\int \frac{d\mu_t(x)}{(w_1-x)^2+w_2^2}}{\left(\int \frac{(w_1-x)d\mu_t(x)}{(w_1-x)^2+w_2^2}\right)^2+w_2^2\left(\int \frac{d\mu_t(x)}{(w_1-x)^2+w_2^2}\right)^2}\nonumber\\
&\geq w_2\times\frac{\int \frac{d\mu_t(x)}{(w_1-x)^2+w_2^2}}{\int \frac{(w_1-x)^2d\mu_t(x)}{\left((w_1-x)^2+w_2^2\right)^2}+w_2^2\int \frac{d\mu_t(x)}{\left((w_1-x)^2+w_2^2\right)^2}}=w_2
\end{align}

\noindent \textbf{Step 2:} We show that $L_z(\C_{\frac{1}{2} \Im(z)})\subset \overline{\C_{\frac{1}{2} \Im(z)}}$ and that $L_z$ is not a conformal automorphism.\\

First, let us show that ${L}_z\left( \mathbb{C}_{\frac{1}{2} \Im(z)} \right) \subset \overline{\mathbb{C}_{\frac{1}{2} \Im(z)}}$. Let $w \in \mathbb{C}_{\frac{1}{2} \Im(z)}$, we have:
		\begin{align}
		\Im \left({L}_z(w)\right) &=  \Im \left[t.G_{\sigma_t} \big( \widetilde{h}_{\mu_t}(w) - z \big) + z \right] \nonumber\\
		&= \Im \left( \dfrac{\widetilde{h}_{\mu_t}(w) - z - \sqrt{\left(\widetilde{h}_{\mu_t}(w) - z\right)^2 - 4t} }{2} + z \right).\label{etape1}
		\end{align}To lower bound the right hand side, note that for all $v\in \C^+$, one can check that:
\begin{align*}
\Im\big(\sqrt{v^2-4t}\big)\leq  &\sqrt{\Im^2(v)+4t}.
\end{align*}Therefore, we have:
		\[ \Im \left(  \sqrt{\big( \widetilde{h}_{\mu_t}(w) - z \big)^2 - 4t} \right) \leq \sqrt{\big[\Im \big( \widetilde{h}_{\mu_t}(w) - z \big) \big]^2 + 4t} .\]
		Hence, \eqref{etape1} yields:
		\begin{align*}
		\Im \big( L_z(w)\big) \geq \dfrac{1}{2} \left[ \Im\big( \widetilde{h}_{\mu_t}(w) - z \big) - \sqrt{\big[\Im \big( \widetilde{h}_{\mu_t}(w) - z \big) \big]^2 + 4t} \right] + \Im(z) .
		\end{align*}The function $g(s) = s - \sqrt{s^2 + 4t}$ is non-decreasing on $\R_+$ and for all $s > 0$, $		g(s) \geq  -2\sqrt{t}$.
		This implies that
		\begin{align}\label{eq:sqrtt}
		\Im\big( {L}_z(w) \big) \geq  \Im(z) -\sqrt{t}> \frac{1}{2} \Im(z),
		\end{align}since $z\in \C_{2\sqrt{t}}$. This guarantees that $L_z(w) \in \mathbb{C}_{\frac{1}{2} \Im(z)}$.		\\

Let us now prove that ${L}_z$ is not an automorphism of $\mathbb{C}_{\frac{1}{2} \Im(z)}$. Consider
		\begin{align*}
		\left| {L}_z(w) - z \right| =  & \left| F_{\sigma_t}\big( \widetilde{h}_{\mu_t}(w) - z \big) - \big( \widetilde{h}_{\mu_t}(w) - z \big) \right| = \left|t G_{\sigma_t}\big(\widetilde{h}_{\mu_t}(w) - z\big)\right|.
		\end{align*}
For $v\in \C^+$, if $|v|>3\sqrt{t}$, since the support of  $\sigma_t$ is $[-2\sqrt{t},2\sqrt{t}]$,
\[
 \left|t G_{\sigma_t}(v)\right| =  \left| \int_{-2\sqrt{t}}^{2\sqrt{t}} \frac{t}{v-x}\dd\sigma_t(x) \right|\leq \sqrt{t}.
 \]
If $|v|\leq 3\sqrt{t}$,
\[\left|t G_{\sigma_t}(v)\right| =\left| \frac{v-\sqrt{v^2-4t}}{2} \right|\leq \frac{2|v|+2\sqrt{t}}{2}\leq 4\sqrt{t}. \]
Hence, for all $w\in \mathbb{C}_{\frac{1}{2} \Im(z)}$,
\begin{equation}\label{eq:ball}
\left| {L}_z(w) - z \right|\leq 4\sqrt{t}.
\end{equation}
This implies that ${L}_z\left( \mathbb{C}_{\frac{1}{2} \Im(z)} \right)$ is included in the ball centered at $z$ with radius $4\sqrt{t}$. As a result, ${L}_z$ is not surjective and hence is not an automorphism of $\mathbb{C}_{\frac{1}{2} \Im(z)}$.\\

\noindent \textbf{Step 3:} Existence and uniqueness of $w_{fp}$, which is a fixed point of $L_z$.\\

By Steps 1 and 2, $L_z$ satisfies the assumptions of Denjoy-Wolff's fixed-point theorem (see e.g. \cite{belinschibercovici,Tarrago1}). The theorem says that for all $w\in \C_{\frac{1}{2} \Im(z)}$ the iterated sequence $L_z^{\circ m}(w)=L_z\circ L_z^{\circ (m-1)}(w)$ converges to the unique Denjoy-Wolff point of $L_z$ which we define as $w_{fp}(z)$. The Denjoy-Wolff point is either a fixed-point of $L_z$ or a point of the boundary of the domain.
Let us check that $w_{fp}$ is a fixed point of $L_z$. For any $z \in \C_{2\sqrt t},$ there exists $\gamma >2$ such that $z \in \C_{\gamma\sqrt t}$ and from \eqref{eq:sqrtt},
$L_z(\C_{\frac{1}{2} \Im(z)}) \subset \C_{(1-\frac{1}{\gamma}) \Im(z)}.$ Moreover, from \eqref{eq:ball}, $L_z(\C_{\frac{1}{2} \Im(z)}) \subset B(z, 4\sqrt t).$ Therefore,
$w_{fp}(z) \in \overline{\C_{(1-\frac{1}{\gamma}) \Im(z)} \cap  B(z, 4\sqrt t) } \subset \C_{\frac{1}{2} \Im(z)},$ so that it is {necessarily} a fixed point.\\

We now define
\[w_1 (z) := F_{\mu_t}(w_{fp}(z) ) + w_{fp}(z) -z.\]

One can check that
\begin{eqnarray*}
 F_{\sigma_t}(w_1(z)) &= &w_1(z) - h_{\sigma_t}(w_1(z)) \\
 & = & F_{\mu_t}(w_{fp}(z) ) + w_{fp}(z) -z -  h_{\sigma_t}(F_{\mu_t}(w_{fp}(z) ) + w_{fp}(z) -z)\\
 & = & \tilde h_{\mu_t}(w_{fp}(z) ) - z - h_{\sigma_t}(\tilde h_{\mu_t}(w_{fp}(z) ) - z ) \\
 & = & \tilde h_{\mu_t}(w_{fp}(z) ) - w_{fp}(z) = F_{\mu_t}(w_{fp}(z)).
\end{eqnarray*}
One can therefore rewrite
\begin{equation*}\label{eq:w1}
 w_1(z) = F_{\sigma_t}(w_1(z)) +  w_{fp}(z)-z.
\end{equation*}
From \eqref{eq:ball} and the fact that $w_{fp}(z)$ is a fixed point of $L_z,$ one easily gets that
\[ \lim_{y\rightarrow +\infty} \frac{w_{fp}(iy)}{iy}=1,\]
which implies that
\[ \lim_{y\rightarrow +\infty} \frac{F_{\mu_t}(w_{fp}(iy))}{iy}=1,\quad \mbox{ and }\quad \lim_{y\rightarrow +\infty} \frac{w_1(iy)}{iy}=1.\]

Now we connect $F_{\mu_0}$ to the previous quantities.
For $z$ large enough, all the functions we consider are invertible and we have
\[ F_{\mu_t}(w_{fp}(z)) +  w_{fp}(z) = z+ w_1 (z) = z + F_{\sigma_t}^{<-1>}(F_{\mu_t}(w_{fp}(z))). \]
On the other hand, for $z$ large enough, using Theorem-definition \ref{def:subordinationfunction} for $\mu_1= \sigma_t$ and  $\mu_2= \mu_0,$ we get
\[ F_{\mu_t}(w_{fp}(z)) +  w_{fp}(z) = \alpha_1(w_{fp}(z)) + \alpha_2(w_{fp}(z))  = F_{\sigma_t}^{<-1>}(F_{\mu_t}(w_{fp}(z))) + F_{\mu_0}^{<-1>}(F_{\mu_t}(w_{fp}(z))).\]
Comparing the two equalities gives
\[F_{\mu_0}^{<-1>}(F_{\mu_t}(w_{fp}(z))) = z, \]
so that, for $z$ large enough,
\[ F_{\mu_t}(w_{fp}(z)) = F_{\mu_0}(z).\]
The two functions being analytic on $\C_{2 \sqrt t},$ the equality can be extended to any $z \in \C_{2 \sqrt t}.$

Finally, since
\[ w_1(z) = F_{\mu_t}(w_{fp}(z) ) + w_{fp}(z) -z=F_{\mu_0}(z) + w_{fp}(z) -z,\]
we have, using \eqref{mino} with $\mu_0$ instead of $\mu_t$,
\[\Im (w_1(z))=\Im(F_{\mu_0}(z)) + \Im(w_{fp}(z)) -\Im(z)\geq  \Im(w_{fp}(z))\geq \frac{1}{2} \Im(z).\]


This ends the proof of Theorem \ref{thm.Fixpoint}.
\subsection{Proof of Theorem-Definition \ref{th:defestimator}}The proof of this theorem follows the steps of the proof of Theorem \ref{thm.Fixpoint}. First, $\widehat{L}_z(w):=t \widehat{G}_{\mu^n_t}(w)+z$ is a well-defined and analytic function on $\C^+$.
Let us check that $\widehat{L}_z\big(\C_{\frac{1}{2}\Im(z)}\big)\subset \C_{\frac{1}{2}\Im(z)}$ for $z\in \C_{2\sqrt{t}}$.
For $w=u+iv \in \C_{\frac{1}{2}\Im(z)}$,
\begin{equation}
\Im\big(\widehat{G}_{\mu^n_t}(w)\big)=\frac{1}{n}\sum_{j=1}^n \Im\Big(\frac{u-\lambda_j^n(t) - i v}{(u-\lambda_j^n(t))^2+v^2}\Big)>-\frac{1}{v}=-\frac{1}{\Im(w)}.\label{eq:ImG}
\end{equation}
Thus,
\begin{align*}
\Im\big(\widehat{L}_z(w)\big)= & t \ \Im\big(\widehat{G}_{\mu^n_t}(w)\big)+\Im(z)>   -\frac{t}{\Im(w)}+\Im(z)
>  -\frac{2t}{\Im(z)}+\Im(z)
>  \frac{1}{2}\Im(z).
\end{align*}The second inequality comes from the choice of $w\in \C_{\frac{1}{2}\Im(z)}$, and the last inequality is a consequence of $\Im(z)>2\sqrt{t}$.\\
Moreover, $\widehat{L}_z$ is not an automorphism since:
\begin{align}\label{Kcircle}\left|\widehat{L}_z(w)-z\right|= & \left|t \widehat{G}_{\mu^n_t}(w)\right|= \left| \frac{1}{n}\sum_{j=1}^n \frac{t}{w-\lambda_j^n(t)}\right| \leq  \frac{t}{\Im(w)}\leq \sqrt{t}
\end{align}
since $\Im(w)> \frac{1}{2}\Im(z)>\sqrt{t}$.
We use again the Denjoy-Wolff fixed-point theorem. Because the inclusion of $\widehat{L}_z\big(\C_{\frac{1}{2}\Im(z)}\big)$ into $ \C_{\frac{1}{2}\Im(z)}$ is strict, the unique Denjoy-Wolff point of $\widehat{L}_z$ is necessarily a fixed point that we denote $\widehat{w}_{fp}(z)$. From the construction, $\Im(\widehat{w}_{fp}(z))>\Im(z)/2$. Finally, the last announced estimate is a straightforward consequence of \eqref{Kcircle}.
\section{Study of the subordination function}\label{section:estimationw}
This section is devoted to the proof of Proposition \ref{propsition.consistency.estimate.w_fp}. We show that $\widehat{w}_{fp}^n(z)$ converges uniformly to $w_{fp}(z)$ on $\C_\gamma$ with $\gamma>2\sqrt{t}$. Next, we establish that the fluctuations are of order $1/\sqrt{n}$.
\subsection{Proof of (i) and (ii) of Proposition \ref{propsition.consistency.estimate.w_fp}}
We first state a useful lemma.
\begin{lem}\label{lem.Lipschitz.CauchyTransform}
For any probability measure $\mu$ on $\R$ and $\alpha>0$, the Cauchy transform $G_{\mu}$ is Lipschitz on $\C_{\alpha}$ with Lipschitz constant $ \dfrac{1}{\alpha^2}$, and one has for any $z\in \C_\alpha$, $\left|G_{\mu}(z)\right|\leq \dfrac{1}{\alpha}$.
\end{lem}

\begin{proof}
	For $z, z' \in \C_{\alpha}$,
	\begin{align*}
	\left| G_{\mu}(z) - G_{\mu}(z') \right| =  & \left| \int_{\R} \dfrac{d\mu(x)}{z-x} - \int_{\R} \dfrac{d\mu(y)}{z'-y} \right| \leq \left|z-z'\right| \ \int_{\R} \dfrac{d\mu(x)}{\left|(z-x)(z'-x)\right|} \\
	&\leq \frac{|z-z'|}{\Im(z)  \Im(z')} \leq \dfrac{|z-z'|}{ \alpha^2 }.
	\end{align*}
	This implies the Lipschitz property of $G_{\mu_t}$.
	Also,
	\begin{align*}
	\left| G_{\mu}(z) \right| = \left| \int_{\R} \dfrac{d\mu(x)}{z-x} \right| \leq \dfrac{1}{\Im(z)} \leq \dfrac{1 }{\alpha} .
	\end{align*}This finishes the proof.
\end{proof}

We are now ready to prove the points (i) and (ii) of Proposition \ref{propsition.consistency.estimate.w_fp}.

\begin{proof}[Proof of Proposition \ref{propsition.consistency.estimate.w_fp}(i-ii)]Consider $z\in \C_{\gamma}$ with $\gamma>2\sqrt{t}$. Using the equations \eqref{eq:lemGmu0} and \eqref{eq:fixed-point.empirical} characterizing $w_{fp}(z)$ and $\widehat w_{fp}^{n}(z)$, we have
	\begin{align}
		\left| \widehat w_{fp}^{n}(z) -  w_{fp}(z) \right| &= t \left| \widehat G_{\mu^{n}_t} ( \widehat w^{n}_{fp}(z)) -  G_{\mu_t} ( w_{fp}(z)) \right| \nonumber  \\
		&\leq    t \left| \widehat G_{\mu^{n}_t} ( \widehat w^{n}_{fp}(z)) -  \widehat{G}_{\mu^n_t} ( w_{fp}(z)) \right| +  t  \left| \widehat{G}_{\mu^n_t} ( w_{fp}(z)) -  G_{\mu_t} ( w_{fp}(z)) \right| .
	\end{align}
By Theorem~\ref{thm.Fixpoint}, $\Im(w_{fp}(z)) \geq \dfrac{1}{2}\Im(z)$ and since $\widehat{G}_{\mu^n_{t}}$ is a Lipschitz function on $\mathbb{C}_{\frac{1}{2} \Im(z)}$ with Lipschitz constant $ \dfrac{4}{\Im^2(z)}\leq \frac{4}{\gamma^2}$, by Lemma \ref{lem.Lipschitz.CauchyTransform}, we have an upper bound for the first term
	\begin{align*}
	\left |\widehat{G}_{\mu^n_t} ( \widehat w^{n}_{fp}(z)) -  \widehat{G}_{\mu^{n}_t} ( w_{f_p}(z)) \right| \leq \dfrac{4}{ \gamma^2} \times \left| \widehat{w}^{n}_{fp}(z) - w_{fp}(z) \right|.
	\end{align*}
	Thus,
	\begin{align*}
	\left| \widehat{w}^{n}_{fp}(z) - w_{fp}(z) \right| \leq \frac{4t}{\gamma^2} \left| \widehat{w}^{n}_{fp}(z) - w_{fp}(z) \right| + t \left| \widehat{G}_{\mu^{n}_t} \big(w_{fp}(z)\big) - G_{\mu_t}\big(w_{fp}(z)\big) \right| ,
	\end{align*}
implying that
	\begin{align}\label{proof.consistency.w^n_fp.upperbound.by.Cauchytransform}
	\left| \widehat{w}^{n}_{fp}(z) - w_{fp}(z) \right| & \leq \left( \dfrac{ t \gamma^2}{\gamma^2 - 4t} \right) \times \left| \widehat{G}_{\mu^{n}_t} \big(w_{fp}(z)\big) - G_{\mu_t}\big(w_{fp}(z)\big) \right|.
	\end{align}
By Proposition \ref{prop:LGN.mu.t}, since the function $x\mapsto \frac{1}{z-x}$ is continuous and bounded on $\R$ for any $z \in \C_{\sqrt{t}}$, $\widehat G_{\mu^n_t}(w_{fp}(z))= \int_\R \frac{1}{w_{fp}(z)-x}\mu^n_t(\dd x)$ converges almost surely to $G_{\mu_t}(w_{fp}(z))= \int_\R \frac{1}{w_{fp}(z)-x}\mu_t(\dd x)$. This concludes the proof of (i).\\

To prove the uniform convergence (ii), we will need Vitali's convergence theorem, see e.g. \cite[Lemma 2.14, p.37-38]{BaiSilver}: on  any bounded compact set of $\C_{2\sqrt{t}}$, the simple convergence is in fact a uniform convergence. Moreover, the functions $G_{\mu_t}(z)$ and $\widehat G_{\mu^n_t}(z)$ decay as $1/|z|$ when $|z|\rightarrow +\infty$, implying the uniform convergence of the right-hand side of \eqref{proof.consistency.w^n_fp.upperbound.by.Cauchytransform} on $\C_\gamma$, for $\gamma>2\sqrt{t}$ and of $\widehat{w}^n_{fp}(z)$ to $w_{fp}(z)$.
\end{proof}



\subsection{Fluctuations of the Cauchy transform of the empirical measure}\label{section:Fluctuations of Cauchy transform}

We now prove point (iii) of Proposition \ref{propsition.consistency.estimate.w_fp}. For this purpose, we first decompose:
\begin{align}
\lefteqn{\widehat{G}_{\mu^{n}_t}(z)-G_{\mu_t}(z)}\nonumber\\
= &  \widehat{G}_{\mu^{n}_t}(z) - \mathbb{E} \big[\widehat{G}_{\mu^{n}_t}(z) | X^n(0)\big] + \mathbb{E} \big[\widehat{G}_{\mu^{n}_t}(z)| X^n(0)\big] -  G_{\mu^{n}_{0} \boxplus \sigma_{t} }(z)+ G_{\mu^{n}_{0} \boxplus \sigma_{t} }(z) -G_{\mu_t}(z)\nonumber\\
=: & A_1^n(z)+A_2^n(z)+A_3^n(z).\label{decomp:G}
\end{align}
The first term is related to the variance of $\widehat{G}_{\mu^{n}_t}(z)$ (conditional on $X^n(0)$). The second term heuristically compares the evolution with the Hermitian Brownian motion to its limit. The third term deals with the fluctuations of the empirical initial condition.
A similar decomposition for the first two terms is done in Dallaporta and F\'evrier \cite{Fevrier1} (without the problem of the random initial condition) and we will adapt their results.
We will show that the fluctuations of the first two terms are of order $1/n$, and this is treated in Propositions \ref{proposition.D&F19.1st.term} and \ref{prop.b_n(z).bound} below. The third term, which is associated to a classical central limit theorem, is of order $1/\sqrt{n}$. This is proved in Proposition \ref{prop.fluctuG2}.\\

For the term $A_1^n(z)$, the result is a direct consequence of Proposition 3 in \cite{Fevrier1} and we refer to the detailed computation in \cite{datthesis}.
\begin{proposition}\label{proposition.D&F19.1st.term}
	For $z \in \C^+$ and $n \in \N$,
\begin{equation*}
		\Var\big(n A^n_1(z)|X^n(0)\big) = \Var\big(n \widehat{G}_{\mu^{n}_t}(z) |X^n(0)\big) \leq \dfrac{10 t}{\Im^4(z)}.
	\end{equation*}
\end{proposition}

\subsubsection{Fluctuations of $A^n_2(z)$}

We start with some additional notations. Let us denote the resolvent of $X^{n}(t)$ by
\begin{equation}\label{eq:Rnt}
	R_{n,t}(z) := \left( z I_{n} - X^{n}(t) \right)^{-1}.
\end{equation}
Then one can write
\begin{equation*}
	\widehat{G}_{\mu^{n}_{t}}(z) = \dfrac{1}{n} \  \mbox{Tr}\left(R_{n,t}(z)\right).
\end{equation*}
Then, the bias term is:
\begin{equation}
 n A^n_2(z)
 =  \mathbb{E}\left[  \Tr  \left(R_{n,t}(z)\right) \ |\ X^n(0)\right]  - n G_{\mu^{n}_{0} \boxplus \sigma_{t}}(z),
 \end{equation}
and it is given by an adaptation of \cite[Proposition 4]{Fevrier1} to the case of a random initial condition:
\begin{proposition}\label{prop.b_n(z).bound}
	For $z \in \C^+$ and $n \in \N$,
\begin{equation}\label{eq:but1}
		\left| n A^n_2(z) \right|  \leq \left( 1 + \frac{4t}{ \Im^2(z)}\right). \left(\dfrac{2t}{\Im^3(z)}  +  \dfrac{12 t^2}{\Im^5(z)} \right).
	\end{equation}
\end{proposition}

The term $A^n_2(z)$ compares $\mathbb{E} \big[\widehat{G}_{\mu^{n}_t}(z)| X^n(0)\big]$ with $G_{\mu_0^n\boxplus \sigma_t}(z)$. Proceeding as in Theorem-Definition \ref{def:subordinationfunction}, with $\mu^n_0$ and $\sigma_t$, we can define a subordination function $\overline{w}_{fp}(z)$ such that
\begin{equation}\label{def:woverline}
G_{\mu_0^n\boxplus \sigma_t}(z)=G_{\mu^n_0}\big(\overline{w}_{fp}(z)\big).
\end{equation}In what follows, it will be natural to introduce and use this subordination function.

\begin{proof}Note that by definition of the resolvent, we have for all $z\in \C^+$,
\begin{equation}\label{et2}|n A^n_2(z)|\leq 2n \Im^{-1}(z),
\end{equation}
which is suboptimal due to the factor $n$.
\\

We follow the ideas of \cite{Fevrier1} for their `approximate subordination relations'. Since our initial condition is random, the strategy has to be adapted and we introduce the following analogues of $R_{n,t}(z)$ and $A^n_2(z)$, which differ from \cite{Fevrier1}:
\begin{align}
 \widetilde{R}_{n,t}(z) & := \Big( \big( z - \frac{t}{n}\mathbb{E} \big[ \Tr \left(R_{n,t}(z)\right) \ |\ X^n(0)\big] \big).I_{n}  - X^{n}(0) \Big)^{-1} \label{eq:Rtildent}\\
 n \widetilde A^n_2(z) & := \mathbb{E} \big[   \Tr  \left(R_{n,t}(z)\right)\ |\ X^n(0)\big]  - \Tr  \big( \widetilde{R}_{n,t}(z) \big).\nonumber
\end{align}
We will bound $A^n_2(z)$ by using its approximation $\widetilde{A}^n_2(z)$.\\

\noindent \textbf{Step 1:} First, we prove an upper bound for $\widetilde{A}^n_2(z)$:
\begin{lem}\label{lem:but1}
For $z \in \C^+$,
	\begin{equation*}
		\big| n \widetilde{A}^n_2(z) \big| \leq \dfrac{2t}{\Im^3(z)}  +  \dfrac{12 t^2}{\Im^5(z)} .
	\end{equation*}
\end{lem}
The proof of this lemma is postponed in Appendix.\\

\noindent \textbf{Step 2:} If $|  \widetilde{A}^n_2(z)|\geq \Im(z)/(2t)$ then, by \eqref{et2}
\[ |n A^n_2(z)| \leq \frac{ 4 t n |\widetilde{A}^n_2(z)|}{\Im^2(z)}.\]	
And we conclude with Lemma \ref{lem:but1}.\\

\noindent \textbf{Step 3:} We now consider the case where $|\widetilde{A}^n_2(z)|< \Im(z)/(2t)$.
We have:
\begin{equation}\label{et13}
 A^n_2(z)=  \widetilde{A}^n_2(z)+ \big[A^n_2(z)-\widetilde{A}^n_2(z)\big]
\end{equation}
We will control the difference $|A^n_2(z)-\widetilde{A}^n_2(z)|$ by $\widetilde{A}^n_2(z)$ and conclude with Lemma \ref{lem:but1}.\\

By their definitions:
\begin{align}\label{et1}
n\big(A^n_2(z)-\widetilde{A}^n_2(z)\big)= &  \Tr  \big( \widetilde{R}_{n,t}(z) \big)-n G_{\mu^n_0\boxplus \sigma_t}(z).
\end{align}
We follow the trick in \cite{Fevrier1} which consists in going back to the fluctuations of the subordination functions. In view of \eqref{def:woverline}, it is natural to express the first term $\Tr  \big( \widetilde{R}_{n,t}(z) \big)$ of \eqref{et1} similarly. As $\widetilde{R}_{n,t}(z)$ is a diagonal matrix,
\begin{equation}\label{et3}
\Tr  \big( \widetilde{R}_{n,t}(z) \big)=\sum_{j=1}^n \frac{1}{z- \frac{t}{n}\E\big[\Tr  \big( R_{n,t}(z) \big)\ |\ X^n(0)\big]-\lambda_j^n(0)}=n G_{\mu^n_0}(\widetilde{w}_{fp}(z)),
\end{equation}where
\begin{equation}\label{def:wwildetilde}\widetilde{w}_{fp}(z):=z- \frac{t}{n}\E\big[\Tr  \big( R_{n,t}(z) \big)\ |\ X^n(0)\big]
\end{equation}and where $\lambda_j^n(0)$ are the eigenvalues of $X^n(0)$. Thus:
\begin{equation}
A^n_2(z)-\widetilde{A}^n_2(z)=   G_{\mu^n_0}(\widetilde{w}_{fp}(z))-G_{\mu^n_0}(\overline{w}_{fp}(z)).\label{et4}
\end{equation}

To continue, we first need the following result proved in Appendix.
\begin{lem}\label{lem:inverse}
(i) The function $\overline{w}_{fp}(z)$, defined in \eqref{def:woverline}, solves
\[\overline{w}_{fp}(z)=z-t G_{\mu^n_0\boxplus \sigma_t}(z).\]
(ii) The function $\zeta(z) =z+tG_{\mu_0^n}(z)$ is well-defined on $\C^+$ and is the inverse of $\overline{w}_{fp}(z)$ on $\overline{\Omega}=\{z\in \C^+,\ \Im(\zeta(z))>0\}$. For such $z\in \overline{\Omega}$, we denote this function $\overline{w}^{<-1>}_{fp}(z)$.
\end{lem}

Let us prove that under the condition of Step 3, $\widetilde{w}_{fp}(z)\in \overline{\Omega}$ for all $z\in \C^+$.
{	\begin{align}
		\zeta(\widetilde{w}_{fp}(z)) - z  &= \widetilde{w}_{fp}(z) + t G_{\mu^{n}_{0}} \big(\widetilde{w}_{fp}(z) \big) - z
		\nonumber\\
		&= z - \frac{t}{n} \mathbb{E} \big[ \Tr  \big( R_{n,t}(z) \big)\  |\ X^n(0) \big]  + t G_{\mu^{n}_{0}} \big( \widetilde{w}_{fp}(z) \big) - z
		\nonumber\\
		&= -t  \widetilde{A}^n_2(z),\label{et5}
	\end{align}}by \eqref{et3}. Therefore,
\begin{align}\label{controlektilde}
		\left|\Im\big(\zeta(\widetilde{w}_{fp}(z)) \big) - \Im(z)\right|  \leq  \left| \zeta(\widetilde{w}_{fp}(z))  -z \right|  =  t \big| \widetilde{A}^n_2(z) \big|  \leq  \frac{\Im(z)}{2}.
	\end{align}
	Thus, under the condition of Step 3, $\widetilde{w}_{fp}(z) \in \overline{\Omega}$. Denoting $\widetilde{z}=\overline{w}^{<-1>}_{fp}\big(\widetilde{w}_{fp}(z)\big)$, which is well-defined, we have $\overline{w}_{fp}(\widetilde{z})=\widetilde{w}_{fp}(z)$. Plugging this into \eqref{et4},
	\begin{align*}
		A^n_2(z)-\widetilde{A}^n_2(z)
		&=   G_{\mu_{0}^{n} \boxplus \sigma_{t}} ( \widetilde{z} )  -  G_{\mu_{0}^{n} \boxplus \sigma_{t}} (z)
		\\
		&=  \big(z-\widetilde{z}\big)  \int_{\R} \dfrac{  \mu_{0}^{n} \boxplus \sigma_{t}( dx) }{ \big(\widetilde{z} - x\big). \big( z - x \big) }
		\\
		&= t \widetilde{A}^n_2(z). \int_{\R} \dfrac{ \mu_{0}^{n} \boxplus \sigma_{t}(dx) }{ \big(\widetilde{z} - x\big). \big( z - x \big) },	\end{align*}where we used \eqref{et5} for the last equality.\\

From there, using \eqref{et13}, we get
	\[ |A^n_2(z) | \le  \Big| 1 + t.\int_{\R} \dfrac{ \mu_{0}^{n} \boxplus \sigma_{t}(dx) }{ \big(\widetilde{z} - x\big). \big( z - x \big) }  \Big|.|\widetilde{A}^n_2(z)|  \leq \left( 1 +  \frac{2t}{\Im^2 (z)} \right) |\widetilde{A}^{n}_2(z)|. \]
This concludes the proof of Proposition~\ref{prop.b_n(z).bound}.
\end{proof}

\subsubsection{Fluctuations of $A^n_3(z)$}

Finally, the third step is to control $A^n_3(z)=G_{\mu^{n}_{0} \boxplus \sigma_{t} }(z) -G_{\mu_t}(z),$ with $\mu_t=\mu_0\boxplus \sigma_t$.

\begin{proposition} \label{prop.fluctuG2}
For any $\gamma>2\sqrt{t}$ and for any $z$ such that $\Im (z) \ge \frac{\gamma}{2},$ we have:
\begin{equation}
 \left|A^n_3(z)\right|<\frac{\gamma^2}{\gamma^2-4t} \left|\int_{\R} \dfrac{1}{z - t.G_{\mu_{0} \boxplus \sigma_{t} }(z) - x} \big[ \dd \mu^{n}_{0}(x) - \dd \mu_{0}(x) \big] \right|
\end{equation}
and
\begin{equation}\label{et8}
\sup_{n\in \N} \sup_{z\in \C_{\frac{\gamma}{2}}} \E\big[n\left|A^n_3(z)\right|^2\big]<\frac{8\gamma^2}{(\gamma^2-4t)^2}.
\end{equation}
\end{proposition}

\begin{proof}
Using again the subordination function $\overline{w}_{fp}(z)$ defined in \eqref{def:woverline} and Lemma~\ref{lem:inverse}(i), we have
 \begin{equation}\label{et6}
 G_{\mu^{n}_{0} \boxplus \sigma_{t} }(z) =   G_{\mu^{n}_{0}} \big( \overline{w}_{fp}(z) \big) =  \int_{\R} \dfrac{\dd \mu^{n}_{0}(x)}{\overline{w}_{fp}(z) - x}   = \int_{\R} \dfrac{ \dd \mu^{n}_{0}(x) }{ z - t G_{\mu^{n}_{0} \boxplus \sigma_{t} }(z) - x}  .
 \end{equation}In this proof, $\Im(z)\geq \gamma/2\geq \sqrt{t}$. Note that $\Im(\overline{w}_{fp}(z))\geq\Im(z)$ (Theorem-Definition \ref{def:subordinationfunction}) so that
 \begin{equation}\label{et7}
 | z - t G_{\mu^{n}_{0} \boxplus \sigma_{t} }(z) - x| \geq \frac{\gamma}{2}\geq \sqrt{t},
 \end{equation}
 and the integrand in \eqref{et6} is well-defined and upper-bounded by $1/\sqrt{t}$.
Similarly, we can establish that
 \begin{equation}
 G_{\mu_0 \boxplus \sigma_{t}}(z) 
 =  \int_{\R} \dfrac{\dd \mu_{0}(x)}{z - t G_{\mu_0 \boxplus \sigma_{t}}(z) - x}.
 \end{equation}
 Then, we can write
 \begin{align*}
 	G_{\mu^{n}_{0} \boxplus \sigma_{t} }(z)  - G_{\mu_0 \boxplus \sigma_{t}}(z) =& \int_{\R} \dfrac{1}{z - t.G_{\mu^{n}_{0} \boxplus \sigma_{t} }(z) - x} \dd \mu^{n}_{0}(x) - \int_{\R} \dfrac{1}{z - t.G_{\mu_{0} \boxplus \sigma_{t} }(z) - x} \dd \mu^{n}_{0}(x)
 	\\
 	&+ \int_{\R} \dfrac{1}{z - t.G_{\mu_{0} \boxplus \sigma_{t} }(z) - x} \dd \mu^{n}_{0}(x) -  \int_{\R} \dfrac{1}{z - t.G_{\mu_0 \boxplus \sigma_{t}}(z) - x} \dd \mu_{0}(x)
 	\\
 	=& \hspace{0.2cm} t. \int_{\R} \dfrac{G_{\mu^{n}_{0} \boxplus \sigma_{t} }(z)- G_{\mu_0 \boxplus \sigma_{t}}(z) }{\left(z - t. G_{\mu^{n}_{0} \boxplus \sigma_{t} }(z) - x \right) . \Big(z - t.G_{\mu_{0} \boxplus \sigma_{t} }(z) - x \Big)} \dd \mu^{n}_{0}(x)
 	\\ &+ \int_{\R} \dfrac{1}{z - t.G_{\mu_{0} \boxplus \sigma_{t} }(z) - x} \left[ \dd \mu^{n}_{0}(x) - \dd \mu_{0}(x) \right].
 \end{align*}
Thus,
 \begin{multline*}
\left( G_{\mu^{n}_{0} \boxplus \sigma_{t} }(z)  - G_{\mu_0 \boxplus \sigma_{t}}(z) \right).\left[ 1  -  t. \int_{\R} \dfrac{ 1 }{\left(z - t.G_{\mu^{n}_{0} \boxplus \sigma_{t} }(z) - x \right) . \Big(z - t.G_{\mu_{0} \boxplus \sigma_{t} }(z) - x \Big)} \dd \mu^{n}_{0}(x) \right]
 	\\
 	=  \int_{\R} \dfrac{1}{z - t.G_{\mu_{0} \boxplus \sigma_{t} }(z) - x} \big[ \dd \mu^{n}_{0}(x) - \dd \mu_{0}(x) \big] .
\end{multline*}
Similarly to \eqref{et7}, we can show that $|z - t.G_{\mu_{0} \boxplus \sigma_{t} }(z) - x| \geq  \gamma/2.$ Thus
\[  \left| t. \int_{\R} \dfrac{ 1 }{\left(z - t.G_{\mu^{n}_{0} \boxplus \sigma_{t} }(z) - x \right) . \Big(z - t.G_{\mu_{0} \boxplus \sigma_{t} }(z) - x \Big)} \dd \mu^{n}_{0}(x) \right| \le \frac{4t}{\gamma^2},\]
consequently,
\[\left|A^n_3(z)\right|<\frac{\gamma^2}{\gamma^2 - 4t}  \left|\int_{\R} \dfrac{1}{z - t.G_{\mu_{0} \boxplus \sigma_{t} }(z) - x} \big[ \dd \mu^{n}_{0}(x) - \dd \mu_{0}(x) \big] \right|,\] which gives the first part of the proposition. For the second part \eqref{et8},

\begin{align*} \E\big[n\left|A^n_3(z)\right|^2\big]
 	\leq \left(\frac{\gamma^2}{\gamma^2 - 4t}\right)^2  n\mathbb E  \Big[\left|\int_{\R} \dfrac{1}{z - t.G_{\mu_{0} \boxplus \sigma_{t} }(z) - x} \big[ \dd \mu^{n}_{0}(x) - \dd \mu_{0}(x) \big] \right|^2\Big].\end{align*}
Now, for any $z$ such that $\Im( z) > \frac{\gamma}{2},$ the function $\varphi_z\ :\ x \mapsto \dfrac{1}{z - t.G_{\mu_{0} \boxplus \sigma_{t} }(z) - x}$ is bounded by $2/\gamma$. Then:
\begin{align*}
n\E\Big[\Big|\int_\R \varphi_z(x)\dd\mu^n_0(x) - \int_\R \varphi_z(x)\dd\mu_0(x)\Big|^2\Big]= & n \E\Big[\Big|\frac{1}{n} \sum_{j=1}^n \varphi_z\big(\lambda^n_j(0)\big)-\E\big[\varphi_z\big(\lambda^n_j(0)\big)\big]\Big|^2\Big]\\
= & n\Var\Big(\frac{1}{n} \sum_{j=1}^n \varphi_z\big(d_j^n\big)\Big)\\
= & \int_{\R} |\varphi_z(x)|^2\dd\mu_0(x)-\Big|\int_{\R} \varphi_z(x)\dd\mu_0(x)\Big|^2\leq   \frac{8}{\gamma^2},
\end{align*}for any $z\in \C_{\frac{\gamma}{2}}$.
\end{proof}

\noindent \textbf{Conclusion:}
We can now conclude the proof of Proposition \ref{propsition.consistency.estimate.w_fp} (iii). From \eqref{decomp:G}, Propositions \ref{proposition.D&F19.1st.term}, \ref{prop.b_n(z).bound} and the first part of Proposition \ref{prop.fluctuG2}, we obtain that for $z\in \C_{\gamma/2}$:
\begin{equation}\E\big[|\widehat{G}_{\mu^{n}_t}(z)-G_{\mu_t}(z)|^2 \  |\ X^n(0)\big] \leq
C(\gamma,t) \Big(\frac{1}{n^2}+ \left|\int_{\R} \dfrac{1}{z - t.G_{\mu_{0} \boxplus \sigma_{t} }(z) - x} \big[ \dd \mu^{n}_{0}(x) - \dd \mu_{0}(x) \big] \right|^2\Big),\label{eq:majoA123conditionnelle}\end{equation}
where $C(\gamma,t)$ depends only on $\gamma$ and $t$ (and converges to $+\infty$ when $\gamma\rightarrow 2\sqrt{t}$). Using the second part of  Proposition \ref{prop.fluctuG2}, we get:
\[\sup_{n\in \N}\sup_{z\in \C_{\gamma/2}} n\E\Big[ |\widehat{G}_{\mu^{n}_t}(z)-G_{\mu_t}(z)|^2 \Big]<+\infty.\]
Equation \eqref{proof.consistency.w^n_fp.upperbound.by.Cauchytransform} implies that for any $\gamma>2\sqrt{t}$,
 \[ \sup_{n\in \N}\sup_{z\in \C_{\gamma} } \mathbb{E} \left[  n \big| \widehat{w}^{n}_{fp}(z) - w_{fp}(z) \big) \big|^{2} \right]\leq  \big(\frac{t\gamma^2}{\gamma^2-4t}\big)^2 \sup_{n\in \N}\sup_{z\in \C_{\frac{\gamma}{2}}}  n\E\Big[ |\widehat{G}_{\mu^{n}_t}(z)-G_{\mu_t}(z)|^2 \Big]< +\infty\]
since if $z\in \C_{\gamma}$ then $\Im\big(w_{fp}(z)\big) \geq \dfrac{1}{2} \Im(z)>\frac{\gamma}{2}$ (Theorem~\ref{thm.Fixpoint}) so that $w_{fp}(z)\in \C_{\frac{\gamma}{2}}$ and point (iii) of Proposition \ref{propsition.consistency.estimate.w_fp} is proved.


\section{Study of  the mean integrated squared error}\label{section:MISE}
In Section~\ref{sec:theo}, we state theoretical results associated with our nonparametric statistical problem. Section~\ref{proof:variance} is devoted to the proof of Theorem \ref{variance}.
\subsection{Theoretical results}\label{sec:theo}
The goal of this section is to study the rates of convergence of $\E\Big[\|\widehat{p}_{0,h}-p_0\|^2\big]$, the mean integrated squared error of $\widehat{p}_{0,h}$.
To derive rates of convergence, we rely on the classical bias-variance decomposition of the quadratic risk. Using Parseval's equality  we obtain
\begin{align}\label{decomp:MISE}
\norm{ \widehat{p}_{0,h} - p_0 }^{2} = \dfrac{1}{2 \pi} \norm{ \widehat{p}_{0,h}^{\star} - p^{\star}_{0} }^{2}
\leq \dfrac{1}{\pi} \norm{ \widehat{p}_{0,h}^{\star} - K^{\star}_{h}.p_0^{\star} }^{2} + \dfrac{1}{\pi} \norm{ K^{\star}_{h}.p_0^{\star}- p_0^{\star} }^{2}.
\end{align}

The expectation of the first term is a variance term whereas the second one is a bias term. To derive the order of the bias term,  we assume that $p_0$ belongs to the space $\mathcal{S}(a,r,L)$ of supersmooth densities defined for $a>0$, $L>0$ and $r>0$ by:
\begin{equation}\label{def:Sobolev}
\mathcal{S}({a,r,L})= \left\{p_0 \; \textrm{density such that} \int_{\R}  |p_0^\star(\xi)|^2 e^{2a |\xi|^r}\dd \xi \leq L \right\}.
\end{equation}
In the literature, this smoothness class of densities has often been considered (see \cite{lacourCRAS}, \cite{ButuceaTsybakov}, \cite{comtelacour}). Most famous examples of supersmooth densities are the Cauchy distribution belonging to $\mathcal{S}({a,r, L})$ with $r=1$   and the Gaussian distribution  belonging to $\mathcal{S}({a,r, L})$ with $r=2$.
To control the bias, we rely on Proposition 1 in  \cite{ButuceaTsybakov} which states that:
\begin{proposition}\label{Biais}
For $p_0 \in \mathcal{S}({a,r,L})$, we have
\[\norm{ K^{\star}_{h}.p_0^{\star}- p_0^{\star} }  \leq C_B L^{1/2} e^{-ah^{-r}},\]
where $C_B$ is a constant.
\end{proposition}
Whereas the control of the bias term is very classical, the study of the variance term in \eqref{decomp:MISE} is much more involved. The order of the variance term is provided by the following theorem.
\begin{thm}\label{variance}
Let
$$\Sigma:=\norm{ \widehat{p}_{0,h}^{\star} - K^{\star}_{h} p_0^{\star} }^{2}.$$ We assume that there exists a constant $C>0$ such that for sufficiently large $\kappa>0$,
\begin{equation}\label{H2}
\mu_0\big((\kappa,+\infty)\big)\leq \frac{C}{\kappa}.
\end{equation}
Then, we have for any $h>0$, for any $\gamma > 2\sqrt t,$
\begin{equation}\label{bound-variance}
	\mathbb{E} (\Sigma) \leq \frac{C_{var}. e^{\frac{2\gamma}{h}}}{n},
\end{equation}
for $C_{var}$ a constant.
\end{thm}
In \eqref{bound-variance}, the constant $C_{var}$ depends on all the parameters of the problem and may blow up when $\gamma$ tends to $2\sqrt{t}$. Theorem~\ref{variance} is proved in Section~\ref{proof:variance}. The main point will be to obtain the optimal $n$ factor appearing at the denominator. The term $e^{\frac{2\gamma}{h}}$ appearing at the numerator is classical in our setting. Note that Assumption~\eqref{H2} is very mild and is satisfied by most classical distributions.\\

Now, using similar computations to those in \cite{lacourCRAS}, we can obtain from Proposition \ref{Biais} and Theorem \ref{variance} the rates of convergence of our estimator $\widehat p_{0,h}$. We indeed showed that:
\begin{equation} \label{bornerisque}
	 MISE:=\E\Big[\left\| \widehat{p}_{0,h} - p_0\right\|^{2}\Big]   \leq   C_{B}^2L  e^{-2ah^{-r}}  +  \frac{C_{var}. e^{\frac{2\gamma}{h}}}{n}.
\end{equation}
Minimizing in $h$ the right hand side of (\ref{bornerisque}) provides the convergence rate of the estimator $\widehat{p}_{0,h}$.
The rates of convergence are summed up in the following corollary, adapted from the computation of \cite{lacourCRAS}. One can see that  there are three cases to consider to derive rates of convergence: $r=1$, $r<1$ and  $r>1$, depending on which the bias or variance term dominates the other. For the sake of completeness  Corollary \ref{thm:MISE} is proved in Appendix.

\begin{corollary}\label{thm:MISE}
Suppose that $\mu_0$ satisfies Assumption~\eqref{H2} and the density $p_0$ belongs to the space $\mathcal{S}(a,r,L)$ for $a>0$, $r>0$ and $L>0$.
Then, for any $\gamma > 2\sqrt t$ and by choosing the bandwidth $h$ according to equation \eqref{tradeoff}, we have:
\begin{equation}
\E\Big[\|\widehat{p}_{0,h}-p_0\|^2\big] = \begin{cases}
O\big(n^{-\frac{a}{a+\gamma}}\big) & \mbox{ if } r=1\\
O\Big(\exp\Big\{-\frac{2a}{(2\gamma)^{r}}\Big[\log n+(r-1)\log \log n+\sum_{i=0}^k b_i^*(\log n)^{r+i(r-1)}\Big]^r\Big\}\Big) & \mbox{ if } r<1\\
O\Big(\frac{1}{n}\exp\Big\{\frac{2\gamma}{(2a)^{1/r}} \Big[\log n+\frac{r-1}{r}\log \log n+\sum_{i=0}^k d_i^* (\log n)^{\frac{1}{r}-i\frac{r-1}{r}}\Big]^{1/r}\Big\}\Big) & \mbox{ if } r>1,
\end{cases}
\end{equation}
where the integer $k$ is such that
\[\frac{k}{k+1}<\min\big(r,\frac{1}{r}\big)\leq \frac{k+1}{k+2},\]
and where the constants $b_i^*$ and $d_i^*$ solve the following triangular system:
\begin{eqnarray*}
b^*_0 & =-\frac{2a}{(2\gamma)^{r}},\qquad \forall i>0,\  b^*_i= & -\frac{2a}{(2\gamma)^{r}}\sum_{j=0}^{i-1} \frac{r(r-1)\cdots (r-j)}{(j+1)!}  \sum_{p_0+\cdots p_j=i-j-1} b^*_{p_0}\cdots b^*_{p_{j}},\\
d^*_0 & =-\frac{2\gamma}{(2a)^{1/r}},\qquad \forall i>0,\ d_i^*= & -\frac{2\gamma}{(2a)^{1/r}} \sum_{j =0}^{i-1} \frac{\frac{1}{r}\big(\frac{1}{r}-1\big)\cdots \big(\frac{1}{r}-j\big)}{(j+1)!} \sum_{p_0+\cdots p_j=i-j-1}  d^*_{p_0}\cdots d^*_{p_j}\end{eqnarray*}
\end{corollary}

 \begin{remark}\label{remarque-fenetre}
For $r=1$, the choice $h=2(a+\gamma)/\log(n)$ yields the rate of convergence. The expressions of the optimal bandwidths for $r>1$ and $r<1$ are much more intricate (see \eqref{eq:hstar1} and \eqref{eq:hstar2} in Appendix, and also \cite{lacourCRAS}).
\end{remark}

Recall that we have transformed the free deconvolution of the Fokker-Planck equation associated with observation of the matrix $X^n(t)$ into the deconvolution problem expressed in (\ref{eq:densiteconvolee}).  To solve the latter, we have then inverted the convolution operator characterized by the Fourier transform of the Cauchy distribution $\mathcal{C}_\gamma$. The parameter $\gamma$ represents the difficulty of our deconvolution problem and consequently, the rates of convergence heavily depend on $\gamma$. The larger $\gamma$ the harder the problem, as can be observed in rates of convergences of  Corollary~\ref{thm:MISE}. This is not surprising: as $t$ grows, it becomes naturally harder to reconstruct the initial condition from the observations at time $t$ and as  $\gamma$ has to be chosen larger than $2\sqrt t,$  $\gamma$ and therefore the difficulty of the deconvolution problem grows with $t$ accordingly. It remains an open question if we can take $\gamma$ smaller. For a given $\gamma$, the upper bound of the variance term given by Theorem~\ref{variance} is optimal. Analogously, the bound for the bias given by Proposition \ref{Biais} is also optimal.  In consequence, rates of convergence for $r=1$ and $r<1$ in Corollary \ref{thm:MISE} are optimal (as proved by Tsybakov in \cite{TsyCRAS} for the case $r=1$ and by Butucea and Tsybakov in \cite{ButuceaTsybakov} for $r<1$). The optimality for $r>1$ remains an open problem.

\subsection{Proof of Theorem~\ref{variance}}\label{proof:variance}
By the definition of $\widehat{p}_{0,h}^\star$, we have:
\begin{align*}
\Sigma&:=\norm{ \widehat{p}_{0,h}^{\star} - K^{\star}_{h} p_0^{\star} }^{2} \\
&= \int_{\R} \dfrac{1}{\pi^2 t^2} e^{2\gamma |\xi| } . \left| K_{h}^{\star}(\xi) \right|^{2} . \left| \left[\left(\Im \big( \widehat{w}^{n}_{fp} (\centerdot+i\gamma )\big) \right)^{\star}  - \left( \Im \big(w_{fp} (\centerdot + i\gamma )\big) \right)^{\star} \right] (\xi)	 \right|^{2} \dd \xi  . \label{eq:MISE.mainTerm1.1}
\end{align*}
Recall that by Lemma \ref{lem:Gmu0}, we have $\Im \left( w_{fp}(z) \right) = t.\Im \big( G_{\mu_t}\big(w_{fp}(z)\big) \big)+ \Im (z)$, and similarly by Theorem-Definition~\ref{th:defestimator}, $\Im \big( \widehat{w}^{n}_{fp}(z) \big) = t.\Im \big( \widehat{G}_{\mu^{n}_t}\big(\widehat{w}^{n}_{fp}(z)\big)\big) + \Im (z)$ for $z \in \C_{2 \sqrt{t}}$.
Since $K^\star_h(\xi)=K^\star(h\xi)$, we have
\begin{align*}
\Sigma &= \int_{\R} e^{2\gamma  |\xi| }. \left| K_{h}^{\star}(\xi) \right|^2. \dfrac{1}{\pi^{2}} \Big| \Big( \Im  \widehat{G}_{\mu^{n}_{t}}\big(\widehat{w}^{n}_{fp}(\centerdot + i\gamma )\big) - \Im G_{\mu_{t}}\big( w_{fp}(\centerdot +  i\gamma )\big) \Big)^{\star}(\xi) \Big|^{2} \dd \xi
\nonumber
\\
& \leq  e^{ \frac{2\gamma}{h}} . \dfrac{C_K^2}{\pi^{2}}. \left\| \Big( \Im  \widehat{G}_{\mu^{n}_{t}}\big(\widehat{w}^{n}_{fp}(\centerdot + i \gamma )\big) - \Im G_{\mu_{t}}\big( w_{fp}(\centerdot + i \gamma )\big) \Big)^{\star} \right\|^{2}
\\
&=  \frac{2C_K^2}{\pi}.   e^{ \frac{2\gamma }{h}}  . \left\| \Im  \widehat{G}_{\mu^{n}_{t}}\big(\widehat{w}^{n}_{fp}(\centerdot +  i\gamma )\big) - \Im G_{\mu_{t}}\big( w_{fp}(\centerdot + i\gamma )\big) \right\|^{2},
\end{align*}by Parseval's equality.
\noindent
Taking the expectation, and introducing a constant $\kappa>0$ chosen later (depending on $n$), we have
\begin{equation}\label{eq:Sigma}
	\mathbb{E}(\Sigma)  \leq \frac{2C_K^2}{\pi}.   e^{ \frac{2\gamma }{h}}  . \big(I^\kappa + J^\kappa)
\end{equation}
where
\begin{align}
I^\kappa=  & \int_{ \left\{ x \in \R : |x| \leq \kappa \right\} } \mathbb{E} \Big[  \Big| \Im  \widehat{G}_{\mu^{n}_{t}}\big(\widehat{w}^{n}_{fp}(x + i\gamma )\big) - \Im G_{\mu_{t}}\big( w_{fp}(x + i\gamma )\big) \Big|^{2}  \Big] \dd x \\
J^\kappa=  & \int_{ \left\{ x \in \R : |x| > \kappa \right\} } \mathbb{E} \Big[  \Big| \Im  \widehat{G}_{\mu^{n}_{t}}\big(\widehat{w}^{n}_{fp}(x + i\gamma )\big) - \Im G_{\mu_{t}}\big( w_{fp}(x + i\gamma )\big) \Big|^{2}  \Big] \dd x.
\end{align}

To obtain the announced rates of convergence for the MISE, we need to be very careful in establishing the upper bounds for $I^\kappa$ and $J^\kappa$. For this purpose, we recall Lemma 4.3.17 of \cite{AGZ}, with a null initial condition, which will be useful in the sequel:
\begin{lem}\label{lem.lambda_j.bounded}
Let  $(\eta^n_1(t), \ldots, \eta^n_n(t))$ be the eigenvalues of $H^n(t).$
With large probability, all the eigenvalues $(\eta^n_j(t))$ of $H^n(t)$ belong to a ball of radius $M>0$ independent of $n$ and $t$. Introduce
\begin{equation}
A_{M}^{n,t} := \left\{ \forall 1 \leq j \leq n : \left|\eta^n_{j}(t)\right| \leq M  \right\} .
\end{equation}
There exist two positive constants $C_{eig}$ and $D_{eig}$ depending on $t$ such that for any $M>D_{eig}$ and any $n\in{\mathbb N}^*$
\begin{equation}
\mathbb{P}\left((A_{M}^{n,t})^c\right) = \mathbb{P}\left(\left\{ \eta^n_*(t) > M\right\}\right) \le    e^{- n.C_{eig}.M},
 \end{equation}
with $\eta_*^n(t):= \max_{i=1, \ldots,n} |\eta^n_i(t)|$.
\end{lem}
Using this lemma, we can control the tail distribution of $\mathbb{E}[\mu^{n}_{t}]$, which is essential to establish very precise estimates without which the announced rate would not be derived. We recall that $\lambda_1^n(t) \le \ldots \le \lambda_n^n(t)$ are the eigenvalues of $X^n(t) = X^n(0) + H^n(t)$ in increasing order. By Weyl's interlacing inequalities, we have that, for $1 \le j \le n,$
\begin{equation}\label{interlacing}
\lambda_j^n(0) - \eta_*^n(t) \le \lambda_j^n(t) \le \lambda_j^n(0) + \eta_*^n(t).
\end{equation}
Therefore, for $1 \le j \le n,$
\[ \mathbb{E} \left[ \mu^{n}_{t}\left( \left\{|\lambda| > \frac{\kappa}{2}\right\}\right) \right]\le  \mathbb{E} \left[ \mu^{n}_{0}\left( \left\{|\lambda| > \frac{\kappa}{4}\right\}\right) \right]
 + n \mathbb{P}\left(\left\{ \eta^n_*(t) > \frac{\kappa}{4}\right\}\right) \le  \mathbb{E} \left[ \mu^{n}_{0}\left( \left\{|\lambda| > \frac{\kappa}{4}\right\}\right) \right]
 + n   e^{- \frac{n.C_{eig}.\kappa}{4}}.\]
Recall that after Equation \eqref{eq:Xn}, we introduced the notation $d_1^n, \ldots, d_n^n$ for the i.i.d. random variables of distribution $\mu_0$ and whose order statistic constitutes the diagonal elements of $X_n(0)$, $\lambda_1^n(0) < \ldots < \lambda_n^n(0)$. We have
\[ \mathbb{E} \left[ \mu^{n}_{0}\left( \left\{|\lambda| > \frac{\kappa}{4}\right\}\right) \right] = \frac{1}{n} \sum_{i=1}^n \mathbb{P}\left( |d_i^n| > \frac{\kappa}{4}\right)
= \mu_{0}\left(\left\{|\lambda | > \frac{\kappa}{4}\right\}\right),
\]
so that we finally get
\begin{equation}
\label{eq:lemmekato}
 \mathbb{E} \left[ \mu^{n}_{t}\left( \left\{|\lambda| > \frac{\kappa}{2}\right\}\right) \right]\le \mu_{0}\left(\left\{|\lambda | > \frac{\kappa}{4}\right\}\right)
 + n   e^{- \frac{n.C_{eig}.\kappa}{4}}.
\end{equation}
Now, we successively study $I^\kappa$ and $J^\kappa$.

\subsubsection{Upper bound  for $I^\kappa$}

\begin{lem}\label{lem:Ikappa}
There exist constants $C_I^2$, $C_I^2$ and $C_I^3$ (that can depend on $\gamma$ and $t$) such that:
\begin{equation}
I^\kappa\leq \frac{C^1_I}{n}+\frac{\kappa C_I^2}{n^2} + C_I^3 \kappa e^{-n.C_{eig}.M}.
\end{equation}
\end{lem}


Before proving Lemma \ref{lem:Ikappa}, let us establish a result that will be useful in the sequel.
\begin{lem}\label{lem:xleqkappa}
Let us consider $\gamma>2\sqrt{t}$, $p>1$ and $M>0$. Then, we have
\begin{equation}
\mathfrak{I}_{p,\gamma,M,t} := \int_0^{+\infty}\int\dfrac{ 1 }{\left[
\left\{\Big|\big| \lambda\big|- x\Big|-\sqrt{t}-M\right\}\vee\frac{\gamma}{2} \right]^p}\dd\mu_0(\lambda)\dd x\leq C(p,\gamma,M,t),
\label{et11}
\end{equation}
for $ C(p,\gamma,M,t)$ a finite constant only depending on $p,$ $\gamma,$ $M$ and $t$.
\end{lem}

\begin{proof}
The supremum in the denominator equals to $\big|  |\lambda|- x\big|-\sqrt{t}-M $ when $x<|\lambda|-\sqrt{t}-M-\gamma/2$ (which is possible only if $|\lambda|-\sqrt{t}-M-\gamma/2$ is positive) or $x>|\lambda|+\sqrt{t}+M+\gamma/2$. Otherwise the supremum is $\gamma/2$. Hence
\begin{align}
\mathfrak{I}_{p,\gamma,M,t} \leq &  \int_\R \left\{ \int_0^{\big(|\lambda|-\sqrt{t}-M-\frac{\gamma}{2}\big) \vee 0} \frac{1}{\Big(|\lambda|-x-\sqrt{t}-M\Big)^p} \dd x + \int^{|\lambda|+\sqrt{t}+M+\frac{\gamma}{2}}_{\{|\lambda|-\sqrt{t}-M-\frac{\gamma}{2}\}\vee 0}\dfrac{2^p}{\gamma^p}\dd x \right. \nonumber\\
 & \hspace{4cm} \left. +\int_{|\lambda|+\sqrt{t}+M+\frac{\gamma}{2}}^{+\infty}\dfrac{ 1 }{\left[x-\big| \lambda\big|-\sqrt{t}-M \right]^p}\dd x \right\} \dd \mu_0(\lambda)\nonumber\\
 \leq &  \int_\R \left\{ \int^{(|\lambda|-\sqrt{t}-M)\vee \frac{\gamma}{2}}_{\frac{\gamma}{2}}\dfrac{ 1}{v^p}\dd v
 +\frac{2^{p}(2\sqrt{t}+2M+\gamma)}{\gamma^p}+ \int_{\frac{\gamma}{2}}^{+\infty} \dfrac{ 1}{v^p}\dd v  \right\} \dd \mu_0(\lambda) \nonumber\\
 \leq & C(p,\gamma,M,t)<+\infty.
\end{align}This concludes the proof.
\end{proof}

\begin{proof}[Proof of Lemma \ref{lem:Ikappa}]
We decompose $I^\kappa$ into three parts, $I^\kappa \leq 2 \big(I^\kappa_1+I^\kappa_2+I^\kappa_3\big)$ where:
\begin{align*}
I^\kappa_{1} &:=  \int_{ \left\{ |x| \leq \kappa \right\} } \mathbb{E} \Big[  \Big|   \widehat{G}_{\mu^{n}_{t}}\big(\widehat{w}^{n}_{fp}(x + i\gamma )\big) -  \widehat{G}_{\mu^{n}_{t}}\big( w_{fp}(x + i\gamma )\big) \Big|^{2}  \Big] \dd x ,\\
I^\kappa_{2} &:=  \int_{ \left\{ |x| \leq \kappa \right\} } \mathbb{E} \Big[  \Big|   \widehat{G}_{\mu^{n}_{t}}\big(w_{fp}(x + i\gamma )\big) -  \mathbb{E} \Big[\widehat{G}_{\mu^{n}_{t}}\big( w_{fp}(x + i\gamma )\big) \ |\ X^n(0)\Big ] \Big|^{2}  \Big] dx,
\\
I^\kappa_{3} &:= \int_{ \left\{ |x| \leq \kappa \right\} } \E\Big[ \Big| \mathbb{E} \Big[\widehat{G}_{\mu^{n}_{t}}\big( w_{fp}(x +  i\gamma )\big) \ |\ X^n(0)\Big] - G_{\mu_{t}}\big( w_{fp}(x + i\gamma )\big) \Big|^{2} \Big] \dd x .
\end{align*}

\noindent \textbf{Step 1:} Let us first upper bound $I^\kappa_{1}$. It is relatively easy to bound $I^\kappa_1$ by an upper bound in $C(\gamma,t) \kappa /n$, but this will not yield in the end the announced convergence rate. To establish more precise upper bounds, we use the event $A_M^{n,t}$ defined in Lemma \ref{lem.lambda_j.bounded}. We have $I^\kappa_{1}= I^\kappa_{11}+I^\kappa_{12}$
with
\begin{align*}
I^\kappa_{11} &:= \int_{ \left\{ |x| \leq \kappa \right\} } \mathbb{E} \bigg[ \Big|   \widehat{G}_{\mu^{n}_{t}}\big(\widehat{w}^{n}_{fp}(x + i\gamma )\big) -  \widehat{G}_{\mu^{n}_{t}}\big( w_{fp}(x + i\gamma )\big) \Big|^{2} 1_{A_{M}^{n,t}} \bigg] \dd x,\\
I^\kappa_{12} &:= \int_{ \left\{ |x| \leq \kappa \right\} } \mathbb{E} \bigg[ \Big|   \widehat{G}_{\mu^{n}_{t}}\big(\widehat{w}^{n}_{fp}(x + i\gamma )\big) -  \widehat{G}_{\mu^{n}_{t}}\big( w_{fp}(x + i\gamma )\big) \Big|^{2} 1_{(A_{M}^{n,t})^c} \bigg] \dd x.
\end{align*}

For the term $I^\kappa_{12}$, we have by Theorem~\ref{thm.Fixpoint}(i) and Lemma \ref{lem.lambda_j.bounded}:
\begin{equation}\label{eq:majoI12}
I^\kappa_{12} \leq \dfrac{16}{\gamma^2}\kappa\mathbb P((A_{M}^{n,t})^c) \leq \dfrac{16}{\gamma^2}\kappa   e^{- n.C_{eig}.M}.
\end{equation}

Let us now consider the term $I^\kappa_{11}$:
\begin{align*}
I^\kappa_{11} &= \int_{ \left\{ |x| \leq \kappa \right\} } \mathbb{E} \bigg[  \bigg| \dfrac{1}{n} \sum_{j=1}^{n} \dfrac{w_{fp}(x + i \gamma) - \widehat{w}^{n}_{fp}(x + i \gamma ) }{ \big( \widehat{w}^{n}_{fp}(x + i \gamma ) - \lambda_{j}^{n}(t) \big) . \big( w_{fp}(x + i \gamma ) - \lambda_{j}^{n}(t) \big) } \bigg|^{2} 1_{A_{M}^{n,t}} \bigg] \dd x
\\
&\leq \int_{ \left\{ |x| \leq \kappa \right\} } \mathbb{E} \Bigg[ \Big| \widehat{w}^{n}_{fp}(x + i \gamma) - w_{fp}(x + i \gamma ) \Big|^{2} . \dfrac{1}{n} \sum_{j=1}^{n} \dfrac{ 1_{A_{M}^{n,t}} }{ \big| \widehat{w}^{n}_{fp}(x + i \gamma ) - \lambda_{j}^{n}(t) \big|^{2} . \big| w_{fp}(x + i \gamma ) - \lambda_{j}^{n}(t) \big|^{2} }  \Bigg] \dd x
\end{align*}
by convexity. Using \eqref{estimee:wfp} and \eqref{interlacing}, we have
\begin{align*}
\big| w_{fp}(x + i \gamma ) - \lambda_{j}^{n}(t) \big|&\geq\big| w_{fp}(x + i \gamma )\big|-\big| \lambda_{j}^{n}(t)\big|\\
&\geq \big|\Re(w_{fp}(x + i \gamma ))\big|-\big| \lambda_{j}^{n}(t)\big|\\
&\geq \big|x\big|-\sqrt{t}-\big| \lambda_{j}^{n}(0)\big|- \eta_*^n(t).
\end{align*}
Since $\lambda_{j}^{n}(t)$ is real, we also have:
\begin{align*}
\big| w_{fp}(x + i \gamma ) - \lambda_{j}^{n}(t) \big|&\geq \big| \Re(w_{fp}(x + i \gamma ) - \lambda_{j}^{n}(t)) \big|\\
&\geq\big|\lambda_{j}^{n}(t) \big|-\big| \Re(w_{fp}(x + i \gamma ))\big|\\
&\geq \big| \lambda_{j}^{n}(t)\big|- \big|x\big|-\sqrt{t}\\
&\geq \big| \lambda_{j}^{n}(0)\big|- \eta_*^n(t)- \big|x\big|-\sqrt{t}.
\end{align*}
Therefore, using Theorem~\ref{thm.Fixpoint},
\begin{equation}
\big| w_{fp}(x + i \gamma ) - \lambda_{j}^{n}(t) \big|\geq\left\{\Big|\big| \lambda_{j}^{n}(0)\big|- \big|x\big|\Big|-\sqrt{t}- \eta_*^n(t)\right\}\vee\frac{\gamma}{2}.
\end{equation}
In Theorem-Definition \ref{th:defestimator}, it is shown that $\widehat{w}^n_{fp}(z)$ satisfies a similar inequality as \eqref{estimee:wfp}. Thus, we obtain with similar computations that:
\begin{equation}
\big| \widehat{w}^{n}_{fp}(x + i \gamma ) - \lambda_{j}^{n}(t) \big|\big|\geq\left\{\Big|\big| \lambda_{j}^{n}(0)\big|- \big|x\big|\Big|-\sqrt{t}- \eta_*^n(t)\right\}\vee\frac{\gamma}{2}.
\end{equation}
Then, using the definition of $A_M^{n,t}$, there exists a constant $C_{11}(\gamma,t)$ only depending on $\gamma$ and $t$ such that
\begin{align}
I^\kappa_{11}
&\leq \int_{ \left\{ |x| \leq \kappa \right\} } \mathbb{E} \Bigg[ \Big| \widehat{w}^{n}_{fp}(x + i \gamma) - w_{fp}(x + i \gamma ) \Big|^{2} . \dfrac{1}{n} \sum_{j=1}^{n} \dfrac{ 1 }{\left[
\left\{\Big|\big| \lambda_{j}^{n}(0)\big|- \big|x\big|\Big|-\sqrt{t}-M\right\}\vee\frac{\gamma}{2} \right]^4}  1_{A_{M}^{n,t}} \Bigg] \dd x\nonumber\\
&\leq \dfrac{1}{n}\sum_{j=1}^{n}  \int_{ \left\{ |x| \leq \kappa \right\} } \mathbb{E} \Bigg[ \dfrac{ 1 }{\left[
\left\{\Big|\big| \lambda_{j}^{n}(0)\big|- \big|x\big|\Big|-\sqrt{t}-M\right\}\vee\frac{\gamma}{2} \right]^4} \mathbb{E} \big[ \big| \widehat{w}^{n}_{fp}(x + i \gamma) - w_{fp}(x + i \gamma ) \big|^{2} |  X^n(0) \big] \Bigg] \dd x\nonumber\\
& \leq \frac{C_{11}(\gamma,t)}{n} \sum_{j=1}^{n} \int_{ \left\{ |x| \leq \kappa \right\} } \mathbb{E} \Bigg[  \dfrac{ 1 }{\left[
\left\{\Big|\big| \lambda_{j}^{n}(0)\big|- \big|x\big|\Big|-\sqrt{t}-M\right\}\vee\frac{\gamma}{2} \right]^4} \left(\frac{1}{n^2} \right. \nonumber\\
& \hspace{5cm} \left.+  \left|\int_{\R} \dfrac{1}{x+i\gamma - t.G_{\mu_{0} \boxplus \sigma_{t} }(x+i\gamma) - \lambda} \big[ \dd \mu^{n}_{0}(\lambda) - \dd \mu_{0}(\lambda) \big] \right|^2\right) \Bigg] \dd x\nonumber\\
 & \leq \frac{C_{11}(\gamma,t)}{n} \big(I^\kappa_{111}+I^\kappa_{112}\big),\label{eq:majoI11}
\end{align}
where the third inequality comes from \eqref{eq:majoA123conditionnelle}, and where:
\begin{align*}
I^\kappa_{111}&:= \dfrac{1}{n^2}\sum_{j=1}^{n} \int_{ \left\{ |x| \leq \kappa \right\} }\mathbb{E} \Bigg[\dfrac{ 1 }{\left[
\left\{\Big|\big| \lambda_{j}^{n}(0)\big|- \big|x\big|\Big|-\sqrt{t}-M\right\}\vee\frac{\gamma}{2} \right]^4} \Bigg]\dd x\\
I^\kappa_{112} & := \int_{ \left\{ |x| \leq \kappa \right\} }\mathbb{E} \Bigg[\int \dfrac{ 1 }{\left[
\left\{\Big|\big| \lambda\big|- \big|x\big|\Big|-\sqrt{t}-M\right\}\vee\frac{\gamma}{2} \right]^4}\dd \mu^{n}_{0}(\lambda) \\
& \left| \sqrt{n}\int_{\R} \dfrac{1}{x+i\gamma - t.G_{\mu_{0} \boxplus \sigma_{t} }(x+i \gamma) - \lambda} \big[ \dd \mu^{n}_{0}(\lambda) - \dd \mu_{0}(\lambda) \big] \right|^2\Bigg]\dd x.
\end{align*}
Now we wish to upper bound $I^\kappa_{111}$ and $I^\kappa_{112} $ independently of $\kappa.$ We first deal with $ I^\kappa_{111}.$
\begin{align}
 I^\kappa_{111}
&= \frac 1 n \int_{ \left\{ |x| \leq \kappa \right\} }\int_\R \dfrac{ 1 }{\left[
\left\{\Big|\big| \lambda\big|- \big|x\big|\Big|-\sqrt{t}-M\right\}\vee\frac{\gamma}{2} \right]^4}\dd\mu_0(\lambda)\dd x\nonumber\\
&\leq \frac 2 n \int_0^{+\infty}\int\dfrac{ 1 }{\left[
\left\{\Big|\big| \lambda\big|- x\Big|-\sqrt{t}-M\right\}\vee\frac{\gamma}{2} \right]^4}\dd\mu_0(\lambda)\dd x. \label{eq:majoI111}
\end{align}The double integral is upper bounded by a constant $C(\gamma,M,t)/n$ by Lemma \ref{lem:xleqkappa}.\\

Let us now consider $I^\kappa_{112}$. Using Cauchy-Schwarz inequality, we have:
\begin{align}
I^\kappa_{112}
& \leq \sqrt{ \mathbb{E} \Bigg[ \int_{ \left\{ |x| \leq \kappa \right\} }\Big(\int \dfrac{ 1 }{\left[
\left\{\Big|\big| \lambda\big|- \big|x\big|\Big|-\sqrt{t}-M\right\}\vee\frac{\gamma}{2} \right]^4}\dd \mu^{n}_{0}(\lambda) \Big)^2 \dd x \Bigg]  }\nonumber\\
& \sqrt{  \mathbb{E} \Bigg[ \int_{ \left\{ |x| \leq \kappa \right\} }\left|\sqrt n \int_{\R} \dfrac{1}{x+i\gamma - t.G_{\mu_{0} \boxplus \sigma_{t} }(x+i \gamma) - \lambda} \big[ \dd \mu^{n}_{0}(\lambda) - \dd \mu_{0}(\lambda) \big] \right|^4\dd x \Bigg]}\label{etape13}
\end{align}
The first term can be treated exactly as $I^\kappa_{111}$ as:
\begin{multline}
 \mathbb{E} \Bigg[ \int_{ \left\{ |x| \leq \kappa \right\} }\Big(\int \dfrac{ 1 }{\left[
\left\{\Big|\big| \lambda\big|- \big|x\big|\Big|-\sqrt{t}-M\right\}\vee\frac{\gamma}{2} \right]^4}\dd \mu^{n}_{0}(\lambda) \Big)^2 \dd x \Bigg] \\
\leq \mathbb{E} \Bigg[ \int_{ \left\{ |x| \leq \kappa \right\} } \frac{1}{\left(\frac{\gamma}{2}\right)^4}\Big(\int \dfrac{ 1 }{\left[
\left\{\Big|\big| \lambda\big|- \big|x\big|\Big|-\sqrt{t}-M\right\}\vee\frac{\gamma}{2} \right]^4}\dd \mu^{n}_{0}(\lambda) \Big) \dd x \Bigg]
=\frac{16n}{\gamma^4}   I^\kappa_{111}.\label{etape14}
\end{multline}
We now focus on the second term of \eqref{etape13}. As in the proof of Proposition \ref{prop.fluctuG2}, if we denote by $\phi_x:=\varphi_{x+i \gamma} : \lambda \mapsto (x+i\gamma - t.G_{\mu_{0} \boxplus \sigma_{t} }(x+i \gamma) - \lambda)^{-1},$ the last term can be rewritten as
\begin{align*}
 I_{1121}^\kappa := & \mathbb{E} \Bigg[ \int_{ \left\{ |x| \leq \kappa \right\} }\left|\sqrt n \int_{\R} \phi_{x}(\lambda) \big[ \dd \mu^{n}_{0}(\lambda) - \dd \mu_{0}(\lambda) \big] \right|^4\dd x \Bigg] \\
  =  & n^2 \int_{ \left\{ |x| \leq \kappa \right\} } \mathbb{E} \Bigg[ \left| \frac{1}{n} \sum_{j=1}^n \left(\phi_{x}(d_j^n(0)) - \mathbb E(\phi_{x}(d_j^n(0))\right)  \right|^4\Bigg] \dd x
  \end{align*}
where we used the notation $d_1^n, \ldots, d_n^n$ for the non-ordered diagonal elements of $X_n(0)$ (introduced after Equation \eqref{eq:Xn}). Since the random variables $d_1^n, \ldots, d_n^n$ are i.i.d. with law $\mu_0,$ the random variables $(\phi_x(\lambda_j^n(0)) - \mathbb E(\phi_x(\lambda_j^n(0)) ))_{1 \le j \le n}$ are i.i.d. centered with finite fourth moment. By Rosenthal and then Cauchy-Schwarz inequality, we have
 \begin{equation}\label{et12}
 I_{1121}^\kappa  \le   \frac{C}{n^2}(n+n^2)   \int_{ \left\{ |x| \leq \kappa \right\}} \int_\R \left|\phi_x(\lambda) - \int_\R \phi_x(\lambda) \dd \mu_{0}(\lambda)  \right|^4  \dd \mu_{0}(\lambda)\dd x,  \end{equation}
 for $C$ a constant. We can conclude if the above double integral is bounded independently of $\kappa$. We would like to use Lemma \ref{lem:xleqkappa} but the fact that we have a non-centered moment here implies that we should be careful, because a constant integrated with respect to $\dd x$ on $\left\{ |x| \leq \kappa \right\}$ yields a term proportional to $\kappa$ that we should avoid.\\

First let us recall some estimates for the functions $\phi_x$. As, we know that $\Im  \big(G_{\mu_0 \boxplus \sigma_t }(x+i \gamma) \big)<0$, we have
\begin{equation}\label{et9}
| x+i\gamma - t G_{\mu_0 \boxplus \sigma_t }(x+i \gamma) - \lambda| \ge \Im ( x+i\gamma - t G_{\mu_0 \boxplus \sigma_t }(x+i \gamma) - \lambda) \ge \gamma \ge \frac{\gamma}{2}\end{equation}
and the functions $\phi_x$ are bounded by $2/\gamma$. This yields that $\big|\int_\R  \phi_x(\lambda) \dd \mu_{0}(\lambda) \big|\leq 2/\gamma$. By Lemma \ref{lem:Gmu0}, $|t  G_{\mu_0 \boxplus \sigma_t }(x+i \gamma) | \le \frac{t}{\gamma} \le \frac{\sqrt t}{2} \le \sqrt t$
so that
\begin{equation}\label{et10} | x+i\gamma - t G_{\mu_0 \boxplus \sigma_t }(x+i \gamma) - \lambda| \ge  (|x - \lambda|- \sqrt t) \geq (\big||x| - |\lambda|\big|- \sqrt t) . \end{equation}
As a consequence,
\begin{equation}
 | x+i\gamma - t G_{\mu_0 \boxplus \sigma_t }(x+i \gamma) - \lambda| \ge (\big||x| - |\lambda|\big|- \sqrt t) \vee \frac{\gamma}{2}.
\end{equation}
Using that $d_1^n$ has distribution $\mu_0$, the double integral in the right hand side of \eqref{et12} can be rewritten as:
\begin{multline*}
\int_{ \left\{ |x| \leq \kappa \right\}} \E\Big( \big|\phi_x(d_1^n) - \E\big[\phi_x(d_1^n)\big]\big|^4\Big)\dd x \\
\begin{aligned}
= & \int_{\left\{ |x| \leq \kappa \right\}}  \Big\{\E\big[ |\phi_x(d_1^n)|^4 \big] - 2 \E\big[|\phi_x(d_1^n)|^2\phi_x(d_1^n)\big]\E\big[\overline{\phi_x}(d_1^n)\big]
+ \E\big[\phi^2_x(d_1^n)\big]\E\big[\overline{\phi_x}(d_1^n)\big]^2 \\
& - 2 \E\big[|\phi_x(d_1^n)|^2\overline{\phi_x}(d_1^n) \big]\E\big[\overline{\phi_x}(d_1^n)\big]
+ 4 \E\big[|\phi_x(d_1^n)|^2\big]\big|\E\big[\phi_x(d_1^n)\big]\big|^2 + \E\big[\overline{\phi_x}(d_1^n)^2\big]\big(\E\big[\phi_x(d_1^n)\big]\big)^2
-  \big|\E\big[\phi_x(d_1^n)\big]\big|^4\Big\}\\
\leq & \mathfrak{I}_{4,\gamma,0,t}+ \frac{8}{\gamma}\mathfrak{I}_{3,\gamma,0,t}+ \frac{24}{\gamma^2} \mathfrak{I}_{2,\gamma,0,t}
\end{aligned}
\end{multline*}
by using the notation of Lemma \ref{lem:xleqkappa} and by neglecting the term $-  \big|\E\big[\phi_x(d_1^n)\big]\big|^4<0$. The Lemma \ref{lem:xleqkappa} allows us to conclude that $ I_{1121}^\kappa\leq C_{1121}(\gamma,t)<+\infty$. \\

We can now conclude the Step 1. This last result, together with \eqref{et12} implies that $I_{112}^\kappa\leq C_{112}(\gamma,t)<+\infty$. From \eqref{eq:majoI11} and \eqref{eq:majoI111}, we have that $I_{11}^\kappa \leq C_1(\gamma,t) /n$ for $C_1(\gamma,t)$ a constant. Gathering this result with \eqref{eq:majoI12}, we finally obtain that:
\begin{equation}\label{eq:majoI1}
I^\kappa_{1}\leq\dfrac{C_1(\gamma,t)}{n}+ \dfrac{16}{\gamma^2}\kappa   e^{- n.C_{eig}.M}.
\end{equation}

\noindent \textbf{Step 2:} Let us consider $I^\kappa_{2}$. Using Proposition~\ref{proposition.D&F19.1st.term}, we have:
\begin{align}
I^\kappa_2= & \int_{ \left\{ |x| \leq \kappa \right\} } \E\Big[\Var \Big( \widehat{G}_{\mu^{n}_{t}}\big( w_{fp}(x + i\gamma)\big) \ |\ X^n(0)\Big) \Big]  \dd x
=   \int_{ \left\{ |x| \leq \kappa \right\} } \E\Big[\Var \Big(  A^n_1\big( w_{fp}(x + i\gamma)\big) \ |\ X^n(0)\Big)  \Big]  \dd x \nonumber\\
 \leq  &   \int_{ \left\{ |x| \leq \kappa \right\} } \dfrac{10 \ t}{n^2 \Im^4\big(w_{fp}(x + i\gamma)\big) } dx
  \leq  \dfrac{10.2^{4}. t . 2\kappa}{n^{2} \gamma^{4}} \leq    \dfrac{20 \kappa}{ n^{2} t}. \label{ett2}
\end{align}

\noindent \textbf{Step 3:} Let us finally provide an upper bound for $I^\kappa_{3}$. Recall the definitions of $A^n_2(z)$ and $A^n_3(z)$ in \eqref{decomp:G}:
\begin{align}
I^\kappa_3 =  &  \int_{\left\{ |x| \leq \kappa \right\}} \E\Big[\big|A^n_2\big(w_{fp}(x+i\gamma)\big)+A^n_3\big(w_{fp}(x+i\gamma)\big)\big|^2 \Big] \dd x\nonumber\\
\leq &  2 \int_{\left\{ |x| \leq \kappa \right\}} \E\Big[\big|A^n_2\big(w_{fp}(x+i\gamma)\big)\big|^2 \Big] \dd x +2 \int_{\left\{ |x| \leq \kappa \right\}} \E\Big[\big|A^n_3\big(w_{fp}(x+i\gamma)\big)\big|^2\Big] \dd x.\label{et14}
\end{align}
By using Proposition \ref{prop.b_n(z).bound} together with Theorem \ref{thm.Fixpoint} (i) and the fact that $\gamma>2\sqrt{t}$, we obtain that the first term in the right hand side is upper-bounded by
\[2 \int_{\left\{ |x| \leq \kappa \right\}} \E\Big[\big|A^n_2\big(w_{fp}(x+i\gamma)\big)\big|^2 \Big]\dd x \leq \frac{c \kappa }{n^2t},\]
where $c$ is an absolute constant. Let us now consider the second term in the right hand side of \eqref{et14}. Using the bound of Proposition \ref{prop.fluctuG2},
\begin{multline}
2 \int_{\left\{ |x| \leq \kappa \right\}} \E\Big[\big|A^n_3\big(w_{fp}(x+i\gamma)\big)\big|^2 \Big] \dd x \\
\leq 2 \frac{\gamma^4}{(\gamma^2-4t)^2}\int_\R \E\Big[\Big|\int_\R \frac{1}{w_{fp}(x+i\gamma) - t. G_{\mu_t}\big(w_{fp}(x+i\gamma)\big)-v}[\dd \mu_0^n(v)-\dd \mu_0(v)]\Big|^2\Big]\dd x.\label{et15}
\end{multline}
Recall that $\mu_0^n$ is the empirical measure of independent random variables $(d_i^n)$ with distribution $\mu_0$ and whose order statistics are the $(\lambda_i^n(0))$. Recalling that $\big(w_{fp}(x+i\gamma)- t.G_{\mu_t}\big(w_{fp}(x+i\gamma)\big)-v\big)^{-1}=\varphi_{w_{fp}(x+i\gamma)}(v),$ we have that
\begin{align}
\E\Big[\Big|\int_\R \varphi_{w_{fp}(x+i\gamma)}(v)[\dd \mu_0^n(v)-\dd \mu_0(v)]\Big|^2\Big]
&=\Var\Big[\frac{1}{n}\sum_{j=1}^n\varphi_{w_{fp}(x+i\gamma)}(\lambda_j^n(0))\Big] \nonumber \\
&\leq \frac{1}{n}\E\big[|\varphi_{w_{fp}(x+i\gamma)}(d_1^n)|^2\big]\nonumber\\
 & = \frac{1}{n}\int_\R  \frac{1}{|w_{fp}(x+i\gamma)- t.G_{\mu_t}\big(w_{fp}(x+i\gamma)\big)-v|^2}
\dd \mu_0(v).\label{et16}
\end{align}
Recall that from Lemma~\ref{lem:Gmu0} and Theorem \ref{thm.Fixpoint} (i), $|w_{fp}(x+i\gamma)- t.G_{\mu_t}\big(w_{fp}(x+i\gamma)\big)-v| \geq |\Im\big(w_{fp}(x+i\gamma)\big)| \geq \gamma/2$, so that the integrand in the right hand side of \eqref{et16} is bounded. However, we have to work more to show that it is integrable with respect to $x$. We have:
\begin{align*}
|w_{fp}(x+i\gamma)- t.G_{\mu_t}\big(w_{fp}(x+i\gamma)\big)-v| &\geq \left|\Re\big(w_{fp}(x+i\gamma)- t.G_{\mu_t}\big(w_{fp}(x+i\gamma)\big)\big)-v\right|\\
&\geq \left|\Re\big(w_{fp}(x+i\gamma)\big)-v\right|-t\left|\Re\left(G_{\mu_t}\big(w_{fp}(x+i\gamma)\big)\right)\right|.
\end{align*}
By Theorem \ref{thm.Fixpoint} (i), we obtain that:
\[\left|\Re\left(G_{\mu_t}\big(w_{fp}(x+i\gamma)\big)\right)\right|\leq\left|\int_\R \frac{\dd\mu_t(y)}{w_{fp}(x+i\gamma)-y}\right|\leq\frac{1}{|\Im(w_{fp}(x+i\gamma))|}\leq\frac{2}{\gamma}.\]
Also, by using \eqref{estimee:wfp}, we get that $|\Re(w_{fp}(x+i\gamma))-x|\leq \sqrt{t}$. Therefore,
\begin{equation}
|w_{fp}(x+i\gamma)- t.G_{\mu_t}\big(w_{fp}(x+i\gamma)\big)-v|\geq \big||x|-|v|\big|-\sqrt{t}-\frac{2t}{\gamma}.
\label{et17}
\end{equation}
From \eqref{et15}, \eqref{et16} and \eqref{et17}, we have that:
\begin{multline*}
2 \int_{\left\{ |x| \leq \kappa \right\}} \E\Big[\big|A^n_3\big(w_{fp}(x+i\gamma)\big)\big|^2 \Big] \dd x \\
\leq \frac{2\gamma^4}{n(\gamma^2-4t)^2} \int_\R \int_\R  \frac{1}{\left( \left\{\big||x|-|v|\big|-\sqrt{t}-\frac{2t}{\gamma}\right\}\vee \frac{\gamma}{2}\right)^2}\dd \mu_0(v)\dd x = \frac{4\gamma^4}{n(\gamma^2-4t)^2} \mathfrak{I}_{2,\gamma,2t/\gamma,t},
\end{multline*}
by Lemma \ref{lem:xleqkappa}.
We conclude as for $I^\kappa_{11}$ and we obtain
\begin{equation}\label{eq:majoI3}
I^\kappa_{3}\leq\dfrac{c \kappa}{n^2t}+ \frac{4\gamma^4}{n(\gamma^2-4t)^2} \mathfrak{I}_{2,\gamma,2t/\gamma,t}.
\end{equation}
Gathering \eqref{eq:majoI1}, \eqref{ett2} and \eqref{eq:majoI3} we obtain the result announced in Lemma~\ref{lem:Ikappa}.
\end{proof}

\subsubsection{Upper bound  for  $J^\kappa$}

Recall the definition of $J^\kappa$ in \eqref{eq:Sigma}. Our goal is to prove the following bound:

\begin{lem}\label{lem:Jkappa}
There exist constants $C_J^1$, $C^2_J$ and $C_J^3$ (that can depend on $\gamma$ and $t$) such that, for any  $\kappa >\gamma,$ we have:
\begin{equation}
J^\kappa\leq \frac{C_J^1}{ \kappa }+ C_J^2 n e^{- \frac{n.C_{eig}.\kappa}{4}} + C_J^3 \mu_{0}\left( \left\{|\lambda| > \frac{\kappa}{4}\right\}\right).
\end{equation}
\end{lem}

\begin{proof}
We decompose $J^\kappa\leq 2 (J^\kappa_1+J^\kappa_2)$ where
\begin{align*}
J^\kappa_1 &:= \int_{ \left\{ |x| > \kappa \right\} }  \mathbb{E} \bigg( \Big| \int_{\R} \dfrac{ \dd \mu^{n}_{t}(\lambda) }{ \widehat{w}^{n}_{fp}(x + i\gamma ) - \lambda }  \Big|^{2} \bigg) \dd x \\
J^\kappa_2 &:= \int_{ \left\{ |x| > \kappa \right\} } \Big| \int_{\R}  \dfrac{ \dd \mu_{t}(\lambda) }{ w_{fp}(x + i \gamma) - \lambda }  \Big|^{2} \dd x.
\end{align*}

Let us consider the first term $J^\kappa_{1}$. Using the estimate of Theorem-Definition \ref{th:defestimator}, we have for all $x\in \R$ that $\big| \Re \big(\widehat{w}^{n}_{fp}(x + i \gamma)\big) - x \big| \leq \sqrt{t}$ and $\Im\big( \widehat{w}^{n}_{fp}(x + i\gamma)\big)\geq \gamma/2.$ This allows us to prove that there exists a constant $C_{asymp}$ such that
\begin{equation}(x-\lambda)^2+\frac{\gamma^2}{4} \leq C_{asymp} \Big(\Re^2\big( \widehat{w}^{n}_{fp}(x + i\gamma) - \lambda \big)^{2} + \frac{\gamma^{2}}{4}\Big).\label{eq:Casympt}\end{equation}
Thus,
\begin{align*}
J^\kappa_1 &\leq  \int_{ \left\{ |x| > \kappa \right\} }  \mathbb{E}  \int_{\R} \dfrac{ \dd \mu^{n}_{t}(\lambda) }{ \Re^{2} \big( \widehat{w}^{n}_{fp}(x + i\gamma) - \lambda \big) + \Im^2\big(\widehat{w}^{n}_{fp}(x + i\gamma ) \big)^{2} }  \dd x
\\
& \leq C_{asymp}  \int_{ \left\{ |x| > \kappa \right\} }  \mathbb{E} \bigg[ \int_{\R} \dfrac{ \dd \mu^{n}_{t}(\lambda) }{  (x -\lambda)^2 + \frac{\gamma^{2}}{4} } \bigg]  \dd x \\
& = C_{asymp}  \, \mathbb{E} \bigg[ \int_{\R} \dd \mu^{n}_{t}(\lambda)  \int_{ \left\{ |x| > \kappa \right\} } \frac{1}{  (x -\lambda)^2 + \frac{\gamma^{2}}{4} } \dd x\bigg] \\
& = \frac{2  C_{asymp}  }{\gamma} \, \mathbb{E} \left[ \int_{\R} \dd \mu^{n}_{t}( \lambda)\left(\pi - \arctan\left(\frac{2}{\gamma}(\kappa-\lambda)\right) -\arctan\left( \frac{2}{\gamma}(\kappa+\lambda)\right)\right)\right] \\
& = \frac{2  C_{asymp}  }{\gamma} \, \mathbb{E} \left[ \int_{\R} \dd \mu^{n}_{t}( \lambda)\left(\arctan\left(\frac{4\kappa \gamma}{4\kappa^2-4\lambda^2 -\gamma^2}\right)
+\pi \mathbf 1_{\left\{\lambda^2 > \kappa^2 - \frac{\gamma^2}{4}\right\}}\right) \right].
\end{align*}
We now use the simple bounds $|\arctan x| \le |x|$ and $|\arctan x| \le \frac{\pi}{2}$ for any $x \in \R.$
Moreover, one can easily check that, if $\lambda^2 \leq \frac{\kappa^2}{2}- \frac{\gamma^2}{4},$ then
\[ \frac{4\kappa \gamma}{4\kappa^2-4\lambda^2 -\gamma^2} \le \frac{2 \gamma}{\kappa}. \]
We therefore get
\begin{align*}
J^\kappa_1 &\leq \frac{2  C_{asymp}  }{\gamma} \, \mathbb{E} \left[ \int_{\lambda} \dd \mu^{n}_{t}( \lambda)\left( \frac{2 \gamma}{\kappa} + \frac{\pi}{2}
  1_{\left\{\lambda^2 > \frac{\kappa^2}{2} - \frac{\gamma^2}{4}\right\}}
+\pi  1_{\left\{\lambda^2 > \kappa^2 - \frac{\gamma^2}{4}\right\}}\right)\right].
\end{align*}
If we assume moreover that $\kappa > \gamma,$ this can be simplified as follows:
\begin{align}
J^\kappa_1 \leq  &  C_{asymp} \left(\frac{4  }{\kappa} + \frac{3\pi }{\gamma}  \mathbb{E} \left[ \mu^{n}_{t}\left( \{|\lambda| > \frac{\kappa}{2}\}\right) \right]\right)\nonumber\\
\leq &   C_{asymp} \left(\frac{4    }{\kappa} + \frac{3\pi  }{\gamma}\mu_{0}\left(\left\{|\lambda | > \frac{\kappa}{4}\right\}\right) +  \frac{3\pi  }{\gamma}
 n   e^{- \frac{n.C_{eig}.\kappa}{4}}\right),\label{eq:Jkappa1}
\end{align}
by using \eqref{eq:lemmekato}.

We now go to the second term $J^\kappa_2.$ The strategy will be very similar to what we did for $J^\kappa_1$ and we will give less details.
 Using the estimate  \eqref{estimee:wfp}, we have for all $x\in \R$ that $\big| \Re \big({w}_{fp}(x + i \gamma)\big) - x \big| \leq \sqrt{t}$, which allows us to get that
\[(x-\lambda)^2+\frac{\gamma^2}{4} \leq C_{asymp} \Big(\Re^2\big( w_{fp}(x + i\gamma) - \lambda \big)^{2} + \frac{\gamma^{2}}{4}\Big),\]
with $C_{asymp}$ as above.
Thus,
\begin{align*}
J^\kappa_2 &\leq  C_{asymp}  \int_{ \left\{ |x| > \kappa \right\} } \int_{\lambda} \dfrac{ \dd \mu_{t}(\lambda) }{  (x -\lambda)^2 + \frac{\gamma^{2}}{4} }   \dd x \\
& \leq \frac{2  C_{asymp}  }{\gamma} \,\int_{\lambda} \dd \mu_{t}( \lambda)\left( \frac{2 \gamma}{\kappa} + \frac{\pi}{2}
  1_{\left\{\lambda^2 > \frac{\kappa^2}{2} - \frac{\gamma^2}{4}\right\}}
+\pi  1_{\left\{\lambda^2 > \kappa^2 - \frac{\gamma^2}{4}\right\}}\right).
\end{align*}
Again, if we assume that $\kappa > \gamma,$ this can be simplified as follows:
\[
J^\kappa_2 \leq C_{asymp} \left(\frac{4    }{\kappa } + \frac{3\pi}{\gamma}  \mu_{t}\left( \left\{|\lambda| > \frac{\kappa}{2}\right\}\right)\right).
\]
Moreover, letting $n$ going to infinity in \eqref{eq:lemmekato}, by Proposition \ref{prop:LGN.mu.t} and dominated convergence, we get that, for any $\kappa > \gamma,$
\[  \mu_{t}\left( \left\{|\lambda| > \frac{\kappa}{2}\right\}\right) \le  \mu_{0}\left( \left\{|\lambda| > \frac{\kappa}{4}\right\}\right), \]
so that
\begin{equation}\label{eq:Jkappa2}
J^\kappa_2
\le C_{asymp} \left(\frac{4    }{\kappa} + \frac{3\pi  }{\gamma}  \mu_{0}\left( \left\{|\lambda| > \frac{\kappa}{4}\right\}\right)\right).
\end{equation}

Gathering the upper bounds  \eqref{eq:Jkappa1} and \eqref{eq:Jkappa2}, we get that for any  $\kappa >\gamma,$ 
\begin{equation}
J^\kappa\leq C_{asympt}\left(\frac{8}{\kappa}+ \frac{6\pi}{\gamma}  \mu_{0}\left( \left\{|\lambda| > \frac{\kappa}{4}\right\} \right)+ \frac{3\pi}{\gamma}ne^{-\frac{n.C_{eig}.\kappa}{4}}\right).\label{eq:Jkappa}
\end{equation}This ends the proof.
\end{proof}

\subsubsection{Conclusion}

As a result, combining  Lemma \ref{lem:Ikappa} and Lemma \ref{lem:Jkappa}, we have:
\begin{align*}
	I^\kappa+J^\kappa \leq & \frac{C^1_I}{n}+ \frac{C_I^2 \kappa}{n^2}+ \frac{C_J^1}{\kappa }+ C_I^3\kappa e^{-n.C_{eig}.M} + C_J^2 n e^{- n C_{eig}.\frac{\kappa}{4}} + C_J^3 \mu_0\left(\left\{ |\lambda| > \frac{\kappa}{4}\right\} \right).
\end{align*}
We take $\kappa=n$. Using Assumption~\eqref{H2}, we obtain
\begin{equation}
\mu_0\left(\left\{ |\lambda| > n\right\} \right)\leq Cn^{-1},
\end{equation}
for some absolute constant $C$. Then, from \eqref{eq:Sigma} and previous computations, there exists a constant $C_{var}$ (that can depend on $\gamma$ and $t$) such that for $n$ sufficiently large:
\begin{equation}
	\mathbb{E} (\Sigma) \leq \frac{C_{var}. e^{\frac{2\gamma}{h}}}{n}
\end{equation}
and Theorem~\ref{variance} is proved.



\section{Numerical simulations}\label{section:numerical}

In this section, we conduct a simulation study to assess the performances of our estimator $\widehat {p}_{0,h}$ designed in Definition~\ref{def:est} based on the $n$-sample $\lambda^n(t):=\{ \lambda_1^n(t),\cdots,\lambda_n^n(t) \}$ of (non ordered) eigenvalues. We consider the sample size $n=4000$ and the time value $t = 1$. We focus on initial conditions following a Cauchy distribution with scale parameter $s_d= 5$: $$p_{0}(x) = \dfrac{1}{\pi}. \dfrac{s_d}{(s_d^2+x^2)},\quad x\in \mathbb{R}.$$ Expression \eqref{Fourier.f_mu0.hat} is used with the kernel $K(x)=\sinc(x)=\sin(x)/(\pi x)$,  and the value $\gamma = 2\sqrt{t}+0.01$ so that the condition $ \gamma > 2 \sqrt{t}$  is satisfied. To implement $\widehat {p}_{0,h}$, we approximate integrals involved in Fourier and inverse Fourier transforms by Riemann sums, so it may happen that $\widehat {p}_{0,h}(x)$ is not real. This is the reason why the density $p_0$ is estimated with $\Re (\widehat {p}_{0,h})$, the real part of $\widehat {p}_{0,h}$.

The theoretical bandwidth $h$ proposed in Section~\ref{section:MISE} cannot be used in practice and
we suggest the following data-driven selection rule, inspired from the principle of cross-validation. We decompose the quadratic risk for $\Re(\widehat{p}_{0,h})$ as follows:
\begin{align*}
	\left\| \Re(\widehat{p}_{0,h}) - p_{0}\right\|^{2} = \int_{\mathbb{R}}  \left|  \Re(\widehat{p}_{0,h}(x)) - p_{0}(x) \right|^{2} dx
	= \left\| \Re(\widehat{p}_{0,h}) \right\|^{2} - 2\int_{\mathbb{R}} \Re( \widehat{p}_{0,h}(x)) p_{0}(x) dx + \left\| p_{0}\right\|^{2}.
\end{align*}
Then, an ideal bandwidth $h$ would minimize the criterion $J$ with
$$
J(h):=\left\| \Re(\widehat{p}_{0,h} )\right\|^{2} - 2\int_{\mathbb{R}}  \Re(\widehat{p}_{0,h}(x))p_{0}(x) dx,\quad h\in\R_+^*.
$$
Since $J$ depends on $p_{0}$ through the second term, we investigate a good estimate of this criterion. For this purpose, we divide the sample $\lambda^n(t)$ into two disjoints sets \[
\mathbf{\lambda}^{n,E}(t) := (\lambda_i^n(t))_{i\in E}\quad\text{ and }\quad \mathbf{\lambda}^{n,E^c}(t) := (\lambda_i^n(t))_{i\in E^c}.
\]
There are $V_{\max} := \binom{n}{n/2}$ possibilities to select the subsets $(E, E^c)$, which is huge. Hence, to reduce  computational time, we draw randomly $V=10$ partitions denoted $(E_j, E_j^c)_{j=1,\ldots,V}$. Choosing the grid $\mathcal{H}$ of $50$ equispaced points lying  between $h_{\min}=0.25$ and $h_{\max}=2.7$, our selected bandwidth is
\begin{align}
	\hat h = \underset{h \in \mathcal{H}}{\textrm{argmin }} \textrm{Crit}(h) \label{def:Crit(h)} 
\end{align}
with
\begin{equation*}
	\textrm{Crit}(h) :=  \min_{h' \in \mathcal{H}, h' \neq h} \dfrac{1}{V} \sum_{j=1}^{V} \left( \left\| \Re(\widehat{p}^{(E_{j})}_{0,h}) \right\|^{2} - 2\int_{\mathbb{R}} \Re( \widehat{p}^{(E_{j})}_{0,h}(x)) \Re( \widehat{p}_{0,h'}^{(E^{c}_{j})} (x)) dx  \right)
\end{equation*}
and our final estimator is then $\Re(\widehat p_{0,\hat h})$. In the last expression, $ \widehat{p}^{(E_{j})}_{0,h}$ and $\widehat{p}_{0,h'}^{(E^{c}_{j})}$ are estimates based on the samples $E_j$ and $E_j^c$ respectively.

To evaluate our approach, Figure \ref{Crit_vs_MISE} displays the plot of $h\in \mathcal{H}\mapsto\textrm{Crit}(h)$ and $h\in \mathcal{H}\mapsto J(h)$  for the Cauchy density  $p_0$.
\begin{figure}[h!]
\begin{center}
\includegraphics[scale=0.3]{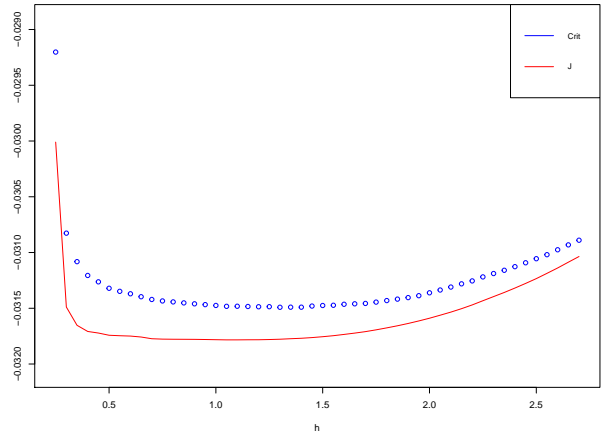}
\end{center}
\caption{Plots of $h\mapsto\textrm{Crit}(h)$ and $h\mapsto J(h)$ for the Cauchy density $p_0$}
\label{Crit_vs_MISE}
\end{figure}
A close inspection of the graphs shows that the first criterion is a good estimate of the second one. As expected, for both criterions, we observe a plateau containing minimizers of $J$ and $\textrm{Crit}$. Outside the plateau, both criterions take large values due to large variance when $h$ is too small and to large bias when $h$ is too large.
\begin{figure}[h!]
\begin{center}
\includegraphics[scale=0.3]{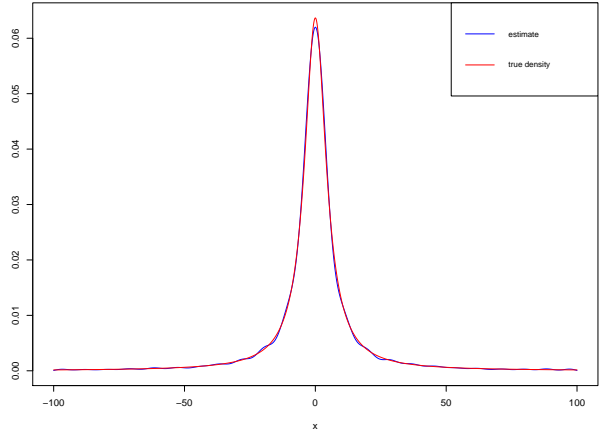}
\end{center}
\caption{Estimation of $p_{0}$}
\label{reconstruction-cauchy}
\end{figure}
Figure~\ref{reconstruction-cauchy} gives the reconstruction provided by $\Re(\widehat p_{0,\hat h})$ for the Cauchy density $p_0$. The results are quite satisfying, meaning that our estimation procedure seems to perform well in practice for estimating initial conditions of the Fokker-Planck equation.

\appendix
\section{Proof of technical lemmas and Corollary \ref{thm:MISE}}
\subsection{Proof of Lemma \ref{lem:but1}}
Recall that $R_{n,t}(z)$ and $\widetilde{R}_{n,t}(z)$ are defined in \eqref{eq:Rnt} and \eqref{eq:Rtildent}, and that
\begin{equation}
n \widetilde{A}^n_2(z) = \sum_{k=1}^n \E\left[ \big(R_{n,t}(z)\big)_{kk}\ |\ X^n(0)\right]-\big(\widetilde{R}_{n,t}(z)\big)_{kk}.\label{etape9}
\end{equation}
Proceeding as in Dallaporta and Février \cite{Fevrier1}, we introduce some notations. Let $R^{(k)}_{n,t}(z)$ be the resolvent of the $(n-1) \times (n-1)$ obtained from $X^{n}(t)$ by removing the $k$-th row and column and $C^{(k)}_{k,t}$ be the $(n-1)$-dimensional vector obtained from the $k$-th column of $H^{n}(t)$ by removing its $k$-th component.\\
Using Schur's complement (see e.g. \cite[Appendix A.1]{BaiSilver}):
\begin{align*}
		\Big(\big( R_{n,t}(z) \big)_{kk}\Big)^{-1} = z - \left(H^{n}(t)\right)_{kk} -  \left(X^{n}(0)\right)_{kk} -  C^{(k) \ast}_{k,t}. R^{(k)}_{n,t}(z). C^{(k)}_{k,t}.
	\end{align*}Because $\widetilde{R}_{n,t}(z)$ is a diagonal matrix, we have easily:
\begin{align*}
		\big( R_{n,t}(z) \big)_{kk} =&  \big( \widetilde{R}_{n,t}(z) \big)_{kk}
		\\
		 &+  \big( \widetilde{R}_{n,t}(z) \big)_{kk}. \big( R_{n,t}(z) \big)_{kk}. \Big(  \left(H^{n}(t)\right)_{kk} +  C^{(k) \ast}_{k,t}. R^{(k)}_{n,t}(z). C^{(k)}_{k,t}  -  \frac{t}{n} \mathbb{E} \big[ \Tr \left(R_{n,t}(z)\ |\ X^n(0)\right) \big]  \Big)   \quad .
	\end{align*}
Replacing $\big( R_{n,t}(z) \big)_{kk}$ in the right-hand side of the previous formula, we obtain:
\begin{multline}
	\left( R_{n,t}(z) \right)_{kk} - \big( \widetilde{R}_{n,t}(z) \big)_{kk} \\
\begin{aligned}
=&  \big( \widetilde{R}_{n,t}(z) \big)_{kk}^{2}. \Big(  \left(H^{n}(t)\right)_{kk} + C_{k,t}^{(k) \ast}. R^{(k)}_{n,t}(z).  C_{k,t}^{(k)} - \frac{t}{n} \mathbb{E}\left[ \Tr \left(R_{n,t}(z)\Big)\ |\ X^n(0) \right] \right)
	\\
	+&  \big( \widetilde{R}_{n,t}(z) \big)_{kk}^{2}. \big( R_{n,t}(z) \big)_{kk}.  \Big(  \left(H^{n}(t)\right)_{kk} + C_{k,t}^{(k) \ast}. R^{(k)}_{n,t}(z). C_{k,t}^{(k)} - \frac{t}{n} \mathbb{E}\left[ \Tr \left(R_{n,t}(z)\right)\ |\ X^n(0) \right] \Big)^{2} .
\end{aligned}\label{etape7}
\end{multline}
Since $H^{n}(t)$ and $C_{k,t}^{(k)}$ are independent of $X_{n}(0)$,
\begin{multline}
\mathbb{E} \Big[ \Big| \left(H^{n}(t)\right)_{kk} + C_{k,t}^{(k) \ast}. R^{(k)}_{n,t}(z) . C_{k,t}^{(k)} - \frac{t}{n} \mathbb{E}\left[ \Tr \left(R_{n,t}(z)\right) \ |\ X^n(0)\right] \Big|^{2}  \, |\,   X^n(0) \Big] \\
	\begin{aligned}
 = &
	\mathbb{E} \Big[ \Big| \left(H^{n}(t)\right)_{kk} + C_{k,t}^{(k) \ast}. R^{(k)}_{n,t}(z) . C_{k,t}^{(k)} - \frac{t}{n}  \Tr \big( R^{(k)}_{n,t}(z) \big) + \frac{t}{n}  \Tr  \big( R^{(k)}_{n,t}(z) \big) - \frac{t}{n} \mathbb{E} \big[ \Tr \big( R^{(k)}_{n,t}(z) \big) \ |\ X^n(0) \big]
	\\ &    + \frac{t}{n}\mathbb{E} \big[ \Tr \big( R^{(k)}_{n,t}(z) \big) \ |\ X^n(0)\big] - \frac{t}{n} \mathbb{E}\left[ \Tr \left(R_{n,t}(z)\right) \ |\ X^n(0)\right] \Big|^{2} \big| X^n(0) \Big]
	\\
	 = &
	\mathbb{E} \left[ \left(H^{n}(t)\right)_{kk}^{2} \right] + \mathbb{E}\Big[ \Big|  C_{k,t}^{(k) \ast}. R^{(k)}_{n,t}(z). C_{k,t}^{(k)} - \frac{t}{n}  \Tr  \big( R^{(k)}_{n,t}(z) \big)  \Big|^{2} \,|\, X^n(0) \Big]
	\\
	& + \frac{t^2}{n^2} \Big(  \textrm{Var} \big[ \Tr \big( R_{n,t}^{(k)}(z) \big) \big| X^n(0) \big]  + \Big| \mathbb{E} \big[ \Tr \big( R_{n,t}^{(k)}(z) \big) - \Tr \left(R_{n,t}(z)\right)\, |\, X^n(0) \big] \Big|^{2} \Big) .
\end{aligned}\label{etape5}
\end{multline}
We now upper bound each of the term in the right-hand side of \eqref{etape5}. The first term equals to $t/n$.\\

\noindent \textbf{Step 1:} We upper bound the second term in \eqref{etape5}. By Lemma 5 of \cite{Fevrier1},
\begin{equation}\label{eq:CRC=trR}
\E\Big[ C_{k,t}^{(k) \ast}. R^{(k)}_{n,t}(z). C_{k,t}^{(k)}  \ |  \ X^n(0)\Big]=\frac{t}{n} \E\Big[\Tr  \big( R^{(k)}_{n,t}(z) \big) \ |\ X^n(0)\Big].
\end{equation}
Thus, the second term in \eqref{etape5} equals to
$\Var\big(C_{k,t}^{(k) \ast}. R^{(k)}_{n,t}(z). C_{k,t}^{(k)} \ | \ X^n(0) \big)$ and we have:

\begin{align*}
\Var\left[ C_{k,t}^{(k) \ast}. R^{(k)}_{n,t}(z). C_{k,t}^{(k)}  \ |  \ X^n(0)\right] =  & \frac{t^2}{n^2} \mathbb{E}\left[\Tr \big(R^{(k),\ast}_{n,t}(z).R^{(k)}_{n,t}(z)\big) \ |\ X^n(0)\right]\\
\leq & \frac{t^2}{n^2} \mathbb{E}\left[ \sum_{j=1}^n \frac{1}{|z-\lambda_j^{(k)} |^2} \ |\ X^n(0)\right]
\end{align*}where the $\lambda_j^{(k)}$'s are the eigenvalues of the matrix with resolvent $R^{(k)}_{n,t}(z)$.
Hence,
\begin{equation}
\Var\left[ C_{k,t}^{(k) \ast}. R^{(k)}_{n,t}(z). C_{k,t}^{(k)}  \ |  \ X^n(0)\right]\leq \frac{t^2}{n \Im^2(z)}.\label{ub2}
\end{equation}

\noindent \textbf{Step 2:} We now upper bound the third and fourth terms of \eqref{etape5}. Let us denote in the sequel by $\mathbb{E}_{k}$ the expectation with respect to $\left\{ \left(H^{n}(t)\right)_{jk} : 1 \leq j \leq n \right\}$, and by $\mathbb{E}_{\leq k}$ the conditional expectation on the sigma-field $\sigma \left( (\left(X^{n}(0)\right)_{ij} , 1 \leq i \leq j \leq n), (\left(H^{n}(t)\right)_{ij}, 1 \leq i \leq j \leq k) \right)$.\\

We have:
\begin{equation}\label{etape8}
\Var \big[ \Tr \big( R_{n,t}^{(k)}(z) \big) \big| X^n(0) \big]\leq 2\Var \big[ \Tr \big( R_{n,t}(z) \big) \big| X^n(0) \big]+ 2 \Var\big[\Tr \big( R_{n,t}(z) \big) - \Tr(R_{n,t}^{(k)}(z) \big| X^n(0) \big].
\end{equation}For the first term,
\begin{align}
\textrm{Var} \big[ \Tr \big( R_{n,t}(z) \big) \big| X^n(0) \big] = & \sum_{k=1}^n \E\left[ \left|  \big(\mathbb{E}_{\leq k} - \mathbb{E}_{\leq k-1} \big) \Tr  \left( R_{n,t}(z) \right) \right|^{2} \ |\ X^n(0)\right]\nonumber\\
= & \sum_{k=1}^n \E\left[ \left|  \big(\mathbb{E}_{\leq k} - \mathbb{E}_{\leq k-1} \big) \big(\Tr  ( R_{n,t}(z) )-\Tr  ( R_{n,t}^{(k)}(z) )\big) \right|^{2} \ |\ X^n(0)\right],\label{etape6}
\end{align}as $\big(\mathbb{E}_{\leq k} - \mathbb{E}_{\leq k-1} \big)\Tr  \left( R_{n,t}^{(k)}(z) \right)=0$.
The Schur complement formula (see e.g. \cite[Appendix A.1]{BaiSilver}) gives that:
	\begin{equation} \label{proof.Var.TrR_n.TrR^k_n}
		\Tr\big( R_{n,t}(z) \big) - \Tr\big( R^{(k)}_{n,t}(z) \big) = \dfrac{1 + C^{(k) \ast}_{k,t}. R^{(k)}_{n,t}(z)^{2}. C^{(k)}_{k,t} }{ z - \left(H^{n}(t)\right)_{kk} - \left(X^{n}(0)\right)_{kk} - C^{(k) \ast}_{k,t}. R^{(k)}_{n,t}(z). C^{(k)}_{k,t} }.
	\end{equation}
Then,	
\begin{align}
	\left| \Tr\big( R_{n,t}(z) \big) - \Tr\big( R^{(k)}_{n,t}(z) \big) \right|  & \leq  \dfrac{ \left| 1 + C^{(k) \ast}_{k,t}. R^{(k)}_{n,t}(z)^{2}. C^{(k)}_{k,t} \right| }{ \left| \Im \left( z - \left(H^{n}(t)\right)_{kk} - \left(X^{n}(0)\right)_{kk} - C^{(k) \ast}_{k,t}. R^{(k)}_{n,t}(z). C^{(k)}_{k,t} \right) \right| }  \nonumber
		\\
		& \leq \dfrac{ 1 + \left| C^{(k) \ast}_{k,t}. R^{(k)}_{n,t}(z)^{2}. C^{(k)}_{k,t} \right| }{\left| \Im(z) - \Im\left( C^{(k) \ast}_{k,t}. R^{(k)}_{n,t}(z). C^{(k)}_{k,t} \right) \right|}  \nonumber		\\
		&\leq  \dfrac{ 1 +  C^{(k) \ast}_{k,t}. R^{(k)}_{n,t}(z)^{\ast}. R^{(k)}(z). C^{(k)}_{k,t}  }{ \left| \Im(z) + \Im\left(z\right) . C^{(k) \ast}_{k,t}. R^{(k)}_{n,t}(z)^{\ast}. R^{(k)}_{n,t}(z). C^{(k)}_{k,t}  \right| }  \nonumber \\
		&= \dfrac{1}{\Im(z)} . \label{proof.R_n.rewrite.bound}
	\end{align}
The second inequality it due to the fact that $\left(H^{n}(t)\right)_{kk}, \left(X^{n}(0)\right)_{kk} \in \R$  and the third inequality comes from the following equality: With $\Psi\ : \ M\in \mathcal{H}_n(\C)\mapsto C^* M C$ with $C\in \C^n$, then, for any $z\in\C$ and any resolvent matrix $R(z)$, we have (see \cite[Lemma 1]{Fevrier1})
\[\Im\big(\Psi(R(z))\big)= -\Im (z) \Psi\big(R(z)^* R(z)\big). \]
The bound \eqref{proof.R_n.rewrite.bound} does not depend on $X^n(0)$. Plugging this bound into \eqref{etape6}, we obtain:
\begin{equation*}
	\textrm{Var} \big[ \Tr \big( R_{n,t}(z) \big) \big| X^n(0) \big] \leq \frac{4n}{\Im^2(z)}.
\end{equation*}From there, using \eqref{etape8},
\begin{equation}\label{ub3}
	\textrm{Var} \big[ \Tr \big( R^{(k)}_{n,t}(z) \big) \big| X^n(0) \big] \leq \frac{8n+2}{\Im^2(z)}.
\end{equation}

Similarly, \eqref{proof.R_n.rewrite.bound} also provides an upper bound for the fourth term of \eqref{etape5}:
\begin{equation}
\Big| \mathbb{E} \big[ \Tr \big( R_{n,t}^{(k)}(z) \big) - \Tr \left(R_{n,t}(z)\right) \ |\ X^n(0) \big] \Big|^{2} \leq \frac{1}{\Im^2(z)}.\label{ub4}
\end{equation}

\noindent \textbf{Step 3:} In conclusion, using \eqref{etape5}, \eqref{ub2}, \eqref{ub3} and \eqref{ub4}, we obtain that:
\begin{multline*}
 \mathbb{E} \left[ \left| \left(H_{n}(t)\right)_{kk} + C_{k,t}^{(k) \ast}. R^{(k)}_{n,t}(z). C_{k,t}^{(k)} - \frac{t}{n} \mathbb{E}\left[ \Tr \left(R_{n,t}(z)\right)\ |\ X^n(0) \right] \right|^{2} \big| X^n(0) \right]
	\\
	\leq  \frac{t}{n}  + \frac{t^2}{n \Im^2(z)}  + \left( 8n + 3 \right)\frac{t^2}{n^2 \Im^2(z)}.
\end{multline*}

Going back to \eqref{etape7} and using \eqref{eq:CRC=trR} to upper-bound the first term in the right-hand side:
\begin{multline*}
	\left| \mathbb{E} \left[  \left( R_{n,t}(z) \right)_{kk} - \big( \widetilde{R}_{n,t}(z) \big)_{kk} \ \big| \ X^{n}(0) \right] \right| \\
\begin{aligned}
	& \leq  \frac{t}{n} \big| \big(\widetilde{R}_{n,t}(z) \big)_{kk} \big|^{2}. \mathbb{E} \left[ \big|  \Tr \big(R^{(k)}_{n,t}(z)\big) - \Tr \big(R_{n,t}(z)\big) \big| \hspace{0.2cm} \Big| \hspace{0.2cm} X^{n}(0) \right]
	 \\
	&  + \big| \big(\widetilde{R}_{n,t}(z) \big)_{kk} \big|^{2}. \mathbb{E} \left[ \left|\big(R_{n,t}(z) \big)_{kk}\right|. \big|\left(H^{n}(t)\right)_{kk} + C_{k,t}^{(k) \ast}. R^{(k)}_{n,t}(z). C_{k,t}^{(k)} \right.\nonumber\\
  & - \left.\frac{t}{n} \mathbb{E}\left[ \Tr \left(R_{n,t}(z)\right)\ |\ X^n(0) \right] \big|^{2} \hspace{0.2cm} \Big| \hspace{0.2cm} X^{n}(0) \right]
	\\
	& \leq \big| \big(\widetilde{R}_{n,t}(z) \big)_{kk} \big|^{2} . \left( \frac{t}{n \Im (z)} + \frac{t}{n \Im (z)} + \frac{t^2}{n \Im^3 (z)} + \frac{(8n+3)t^2}{n^2 \Im^3 (z)} \right) \\
	& \leq \big| \big(\widetilde{R}_{n,t}(z) \big)_{kk} \big|^{2} . \frac{1}{n}\left( \frac{2t}{\Im (z)}+ \frac{12 t^2}{\Im^3 (z)} \right) .
\end{aligned}
\end{multline*}
Using this upper bound in \eqref{etape9}, we obtain by summation the result and using that for any $k$,
\[\big|\widetilde{R}_{n,t}(z) \big)_{kk} \big|^{2}\leq \frac{1}{\Im^2(z)}.\]
\subsection{Proof of Lemma \ref{lem:inverse}}
From \eqref{def:woverline} and introducing $\overline{w}_1(z)$ such that:
\[G_{\mu_0^n\boxplus \sigma_t}(z)=G_{\mu^n_0}\big(\overline{w}_{fp}(z)\big)=G_{\sigma_t}(\overline{w}_1(z)).\]
We can derive from Theorem-Definition~\ref{def:subordinationfunction} that $\overline{w}_{fp}(z)$ solves the equation (i) of Lemma \ref{lem:inverse} and that:
\[z=\overline{w}_{fp}(z)+t G_{\mu_0^n}(\overline{w}_{fp}(z)),\]
for all $z\in \C^+$. The latter equation justifies (ii) of Lemma \ref{lem:inverse}.

\subsection{Proof of Corollary \ref{thm:MISE}}

Recall that from Proposition \ref{Biais} and Theorem \ref{variance}, the mean integrated square error is
\[
	 MISE=\E\Big[\left\| \widehat{p}_{0,h} - p_0\right\|^{2}\Big]   \leq   C_{B}^2L  e^{-2ah^{-r}}  +  \frac{C_{var}. e^{\frac{2\gamma}{h}}}{n}.
\]
Minimizing in $h$ amounts to solving the following equation obtained by taking the derivative in the right hand side of (\ref{bornerisque}):
\begin{equation}\label{tradeoff}
\psi(h):= \exp{(\frac{2\gamma}{h} +\frac{2a}{h^r}) } h^{r-1}= O(n).
\end{equation}
Consequently for the minimizer $h_*$ of \eqref{tradeoff} we get  that
\[
\frac{e^{\frac{2\gamma}{h_*}}}{n}=Ch_*^{1-r} e^{-2ah_*^{-r}},
\]
for some constant $C>0$. Hence, in view of \eqref{bornerisque}, when $r<1$ the bias dominates the variance and the contrary occurs when $r>1$. Thus, there are three cases to consider to derive rates of convergence: $r=1$, $r<1$ and  $r>1$. To solve the equation \eqref{tradeoff}, we follow the steps of Lacour \cite{lacourCRAS}.\\

\noindent \textbf{Case $r=1$.} \\

The case where $r=1$ provides a window $h_{*}=2(a+\gamma)/\log n $  and we get
\[MISE = O \left (n^{-\frac{a}{a+\gamma}} \right ).\]

\noindent \textbf{Case $r<1$.} \\

In this case, and in the case $r>1$, following the ideas in \cite{lacourCRAS}, we will look for the bandwidth $h$ expressed as an expansion in $\log (n)$. In this expansion and when $r<1$, the integer $k$ such that $\frac{k}{k+1}<r\leq \frac{k+1}{k+2}$ will play a role. The optimal bandwidth is of the form:
\begin{equation}\label{eq:hstar1}h_*= 2\gamma \Big(\log(n) + (r-1) \log \log (n) + \sum_{i=0}^k b_i (\log n )^{r+i(r-1)}\Big)^{-1},\end{equation}
where the coefficients $b_i$'s are a sequence of real numbers chosen so that $\psi(h_*)=O(n)$. The heuristic of this expansion is as follows: the first term corresponds to the solution of $e^{2\gamma/h}=n$. The second term is added to compensate the factor $h^{r-1}$ in \eqref{tradeoff} evaluated with the previous bandwidth, and the third term aims at compensating the factor $e^{2a/h^r}$. Notice that $r-1<0$ and that the definition of $k$ implies that $r>r+(r-1)>\dots >r+k(r-1)>0>r+(k+1)(r-1)$. This explains the range of the index $i$ in the sum of the right hand side of \eqref{eq:hstar1}.\\

Plugging \eqref{eq:hstar1} into \eqref{tradeoff},
\begin{align*}
\psi(h_*)
 = & n \big(\log n\big)^{r-1}\exp\Big( \sum_{i=0}^k b_i (\log n)^{r+i(r-1)}\Big)\\
   & \times
\exp\Big(\frac{2a}{(2\gamma)^r} \big(\log n\big)^r \big(1 + \frac{(r-1) \log \log (n) + \sum_{i=0}^k b_i (\log n)^{r+i(r-1)}}{\log n}\big)^r \Big)\\
  & \times (2\gamma)^{r-1} \big(\log n\big)^{-(r-1)} \Big(1 + \frac{(r-1) \log \log (n) + \sum_{i=0}^k b_i (\log n)^{r+i(r-1)}}{\log n}\Big)^{-(r-1)}\\
 = & (2\gamma)^{r-1} n(1+v_n)^{1-r}\exp\Big( \sum_{i=0}^k b_i (\log n)^{r+i(r-1)}\Big)\\
  & \times \exp\Big(\frac{2a}{(2\gamma)^r} \big(\log n\big)^r \Big[1+ \sum_{j=0}^{k} \frac{r(r-1)\cdots (r-j)}{(j+1)!}v_n^{j+1}+o(v_n^{k+1})\Big]\Big)
\end{align*}where
\[v_n=\frac{(r-1) \log \log (n) + \sum_{i=0}^k b_i (\log n)^{r+i(r-1)}}{\log n}=(r-1)\frac{\log \log(n)}{\log n}+\sum_{i=0}^k b_i (\log n)^{(i+1)(r-1)}\]
converges to zero when $n\rightarrow +\infty$. We note that
\begin{align*}
v_n^{j+1}=  & \sum_{i=0}^{k-j-1} \sum_{p_0+\cdots p_{j}=i}b_{p_0}\cdots b_{p_{j}} (\log n)^{(i+j+1)(r-1)}+O\Big(\big(\log n\big)^{(k+1)(r-1)}\Big)\\
= & \sum_{\ell=j+1}^{k} \sum_{p_0+\cdots p_j=\ell-j-1} b_{p_0}\cdots b_{p_{j}} (\log n)^{\ell(r-1)}+O\Big(\big(\log n\big)^{(k+1)(r-1)}\Big).\end{align*}
So
\begin{align*}
\psi(h_*)= & (2\gamma)^{r-1} n (1+v_n)^{1-r}\exp\Big( \sum_{i=0}^k b_i (\log n)^{r+i(r-1)}\Big)\\
 &  \times \exp\Big\{\frac{2a}{(2\gamma)^r}(\log n\big)^r + \frac{2a}{(2\gamma)^r}\sum_{\ell=1}^k \sum_{j=0}^{\ell-1} \Big[\frac{r(r-1)\cdots (r-j)}{(j+1)!}\sum_{p_0+\cdots p_j=\ell-j-1} b_{p_0}\cdots b_{p_{j}} \Big](\log n)^{r+\ell(r-1)}\\
& +O\Big(\big(\log n\big)^{(k+1)(r-1)}\Big)\Big\} \\
 = &
(2\gamma)^{r-1} n (1+v_n)^{1-r} \exp\Big(\sum_{i=0}^{k}  M_i (\log n)^{i(r-1)+r} + o(1)\Big).
\end{align*}The condition $\psi(h_*)=O(n)$ implies the following choices of constants $M_i$'s:
\[
M_0=  b_0+\frac{2a}{(2\gamma)^{r}},\qquad \forall i>0,\
M_i= b_i+\frac{2a}{(2\gamma)^{r}}\sum_{j=0}^{i-1} \frac{r(r-1)\cdots (r-j)}{(j+1)!}  \sum_{p_0+\cdots p_j=i-j-1} b_{p_0}\cdots b_{p_{j}}.
\]
Since $h_*$ solves \eqref{tradeoff} if all the $M_i=0$ for $i\in \{0,\cdots k\}$, the above system provides equation by equation the proper coefficients $b_i^*$.
\begin{equation}\label{eq:bstar}
b^*_0=-\frac{2a}{(2\gamma)^{r}},\qquad b^*_i=-\frac{2a}{(2\gamma)^{r}}\sum_{j=0}^{i-1} \frac{r(r-1)\cdots (r-j)}{(j+1)!}  \sum_{p_0+\cdots p_j=i-j-1} b^*_{p_0}\cdots b^*_{p_{j}}.\end{equation}
Replacing in \eqref{bornerisque}, we get:
\[MISE=O\Big(\exp\Big\{-\frac{2a}{(2\gamma)^{r}}\Big[\log n+(r-1)\log \log n+\sum_{i=0}^k b_i^*(\log n)^{r+i(r-1)}\Big]^r\Big\}\Big).\]

\noindent \textbf{Case $r>1$.}\\

Here, let us denote by $k$ the integer such that $\frac{k}{k+1}<\frac{1}{r}\leq \frac{k+1}{k+2}$. We look here for a bandwidth of the form:
\begin{equation}\label{eq:hstar2}h_*^r= 2a \Big(\log n + \frac{r-1}{r} \log \log (n) + \sum_{i=0}^k d_i (\log n)^{\frac{1}{r}-i\frac{r-1}{r}}\Big)^{-1},\end{equation}
where the coefficients $d_i$'s will be chosen so that $\psi(h_*)=O(n)$.

Similar computations as for the case $r<1$ provide that:
\begin{align*}
\psi(h_*)= & (2a)^{\frac{r-1}{r}} n (1+v_n)^{-\frac{r-1}{r}}\times \exp\Big(\sum_{i=0}^k d_i (\log n)^{\frac{1}{r}-i \frac{r-1}{r}}\Big)\\
 & \times \exp\Big(\frac{2\gamma}{(2a)^{1/r}} (\log n)^{1/r} \Big[1 +\\
 & \quad \sum_{\ell=1}^{k} \sum_{j =0}^{\ell-1} \sum_{p_0+\cdots p_j=\ell-j-1} \frac{\frac{1}{r}\big(\frac{1}{r}-1\big)\cdots \big(\frac{1}{r}-j\big)}{(j+1)!} d_{p_0}\cdots d_{p_j} (\log n)^{\ell \frac{1-r}{r}}  +O\big((\log n)^{k\frac{1-r}{r}}\big)\Big]\Big)\\
 = & (2a)^{\frac{r-1}{r}} n (1+v_n)^{-\frac{r-1}{r}} \exp\Big(\sum_{i=0}^{k} M_i (\log n)^{\frac{1}{r}-i\frac{r-1}{r}}+o(1)\Big)
\end{align*}
where here
\[v_n=\frac{\frac{r-1}{r} \log \log (n)+\sum_{i=0}^k d_i (\log n)^{\frac{1}{r}-i\frac{r-1}{r}}}{\log n},\]
and
\begin{equation}\label{eq:dstar}
M_0=d_0+\frac{2\gamma}{(2a)^{1/r}},\qquad \forall i>0,\ M_i=d_i+\frac{2\gamma}{(2a)^{1/r}} \sum_{j =0}^{i-1} \sum_{p_0+\cdots p_j=i-j-1} \frac{\frac{1}{r}\big(\frac{1}{r}-1\big)\cdots \big(\frac{1}{r}-j\big)}{(j+1)!} d_{p_0}\cdots d_{p_j}\end{equation}
Solving $M_0=\cdots =M_k=0$ provides the coefficients $d_i^*$ so that \eqref{tradeoff} is satisfied.\\

Plugging the bandwidth $h_*$ with the coefficients $d_i^*$ into \eqref{bornerisque}, we obtain:
\[MISE=O\Big(\frac{1}{n}\exp\Big\{\frac{2\gamma}{(2a)^{1/r}} \Big[\log n+\frac{r-1}{r}\log \log n+\sum_{i=0}^k d_i^* (\log n)^{\frac{1}{r}-i\frac{r-1}{r}}\Big]^{1/r}\Big\}\Big).\]
This concludes the proof of Corollary \ref{thm:MISE}.



\subsubsection*{Acknowledgement}
The authors thank P. Tarrago for useful discussions. M.M. acknowledges support from the Labex CEMPI (ANR-11-LABX-0007-01). V.C.T. is partly supported by Labex B\'ezout (ANR-10-LABX-58) and by the Chair ``Modélisation Mathématique et Biodiversité" of Veolia Environnement-Ecole Polytechnique-Museum National d'Histoire Naturelle-Fondation X.

{\footnotesize
\providecommand{\noopsort}[1]{}\providecommand{\noopsort}[1]{}\providecommand{\noopsort}[1]{}\providecommand{\noopsort}[1]{}

}

\begin{thebibliography}{10}

\bibitem{AGZ}
G.W. Anderson, A.~Guionnet, and O.~Zeitouni.
\newblock {\em An Introduction to Random Matrices}, volume 118 of {\em
  Cambridge Studies in Advanced Mathematics}.
\newblock Cambridge University Press, Cambridge, 2010.

\bibitem{Tarrago1}
O.~Arizmendi, P.~Tarrago, and C.~Vargas.
\newblock Subordination methods for free deconvolution.
\newblock submitted. \verb"arxiv:1711.08871", 2020.

\bibitem{BaiSilver}
Z.~Bai and J.~W. Silverstein.
\newblock {\em Spectral Analysis of Large Dimensional Random Matrices},
  volume~37 of {\em Series in Statistics}.
\newblock Springer, 2 edition, 2006.

\bibitem{belinschibercovici}
S.~Belinschi and H.~Bercovici.
\newblock A new approach to subordination results in free probability.
\newblock {\em Journal d'Analyse Mathematique}, 101:357--365, 2007.

\bibitem{bertoingiraudisozaki}
J.~Bertoin, C.~Giraud, and Y.~Isozaki.
\newblock Statistics of a flux in {B}urgers turbulence with one-sided
  {B}rownian initial data.
\newblock {\em Comm. Math. Phys.}, 224(2):551--564, 2001.

\bibitem{Biane1}
P.~Biane.
\newblock On the free convolution with a semi-circular distribution.
\newblock {\em Indiana University Mathematics Journal}, 46(3):705--718, 1997.

\bibitem{bianespeicher}
P.~Biane and R.~Speicher.
\newblock Free diffusions, free entropy and free {F}isher information.
\newblock {\em Ann. Inst. H. Poincar\'{e} Probab. Statist.}, 37(5):581--606,
  2001.

\bibitem{bourgain}
J.~Bourgain.
\newblock Periodic nonlinear {S}chr\"{o}dinger equation and invariant measures.
\newblock {\em Comm. Math. Phys.}, 166(1):1--26, 1994.

\bibitem{burgers}
J.M. Burgers.
\newblock {\em The Nonlinear Diffusion Equation}.
\newblock Springer, 1974.

\bibitem{burqtzevtkovI}
N.~Burq and N.~Tzvetkov.
\newblock Random data cauchy theory for supercritical wave equations i: Local
  theory.
\newblock {\em Inventiones Mathematicae}, 173:449--475, 2008.

\bibitem{burqtzevtkovII}
N.~Burq and N.~Tzvetkov.
\newblock Random data cauchy theory for supercritical wave equations ii: A
  global result.
\newblock {\em Inventiones Mathematicae}, 173:477--496, 2008.

\bibitem{ButuceaTsybakov}
C.~Butucea and A.~B. Tsybakov.
\newblock Sharp optimality in density deconvolution with dominating bias. {I}.
\newblock {\em Teor. Veroyatn. Primen.}, 52(1):111--128, 2007.

\bibitem{carrillomccannvillani}
J.A. Carrillo, R.J. McCann, and C.~Villani.
\newblock Kinetic equilibration rates for granular media and related equations:
  entropy dissipation and mass transportation estimates.
\newblock {\em Rev. Mat. Iberoamericana}, 19(3):971--1018, 2003.

\bibitem{comtelacour}
F.~Comte and C.~Lacour.
\newblock Anisotropic adaptive kernel deconvolution.
\newblock {\em Ann. Inst. H. Poincar\'e Probab. Statist.}, 49(2):569--609,
  2013.

\bibitem{constantinwu}
P.~Constantin and J.~Wu.
\newblock Statistical solutions of the {N}avier-{S}tokes equations on the phase
  space of vorticity and the inviscid limit.
\newblock {\em Journal of Mathematical Physics}, 38(6):3031--3045, 06 1997.

\bibitem{Fevrier1}
S.~Dallaporta and M.~Fevrier.
\newblock Fluctuations of linear spectral statistics of deformed wigner
  matrices.
\newblock submitted. \verb"hal-02079313", 2019.

\bibitem{flandoli2018}
F.~Flandoli.
\newblock Weak vorticity formulation of 2{D} {E}uler equations with white noise
  initial condition.
\newblock {\em Comm. Partial Differential Equations}, 43(7):1102--1149, 2018.

\bibitem{giraud2003}
C.~Giraud.
\newblock Some properties of burgers turbulence with white noise initial
  conditions.
\newblock In {\em Probabilistic Methods in Fluids}, pages 161--178. World
  Scientific, 2003.

\bibitem{lacourCRAS}
C.~Lacour.
\newblock Rates of convergence for nonparametric deconvolution.
\newblock {\em Comptes rendus de l'Acad\'emie des sciences. S\'erie I,
  Math\'ematique}, 342(11):877--882, 2006.

\bibitem{meleardcime}
S.~M\'{e}l\'{e}ard.
\newblock Asymptotic behaviour of some interacting particle systems,
  {M}c{K}ean-{V}lasov and {B}oltzmann models.
\newblock In {\em CIME Lectures}, volume 1627 of {\em Lecture Notes in
  Mathematics}, pages 45--95. Springer, 1996.

\bibitem{datthesis}
T.D. Nguyen.
\newblock {\em Statistical deconvolution of {F}okker-{P}lanck equation}.
\newblock PhD thesis, Universit\'e Paris Saclay, Paris, France, 2021.

\bibitem{penskysapatinas09}
M.~Pensky and T.~Sapatinas.
\newblock Functional deconvolution in a periodic setting: uniform case.
\newblock {\em Ann. Statist.}, 37(1):73--104, 2009.

\bibitem{penskysapatinas10}
M.~Pensky and T.~Sapatinas.
\newblock On convergence rates equivalency and sampling strategies in
  functional deconvolution models.
\newblock {\em Ann. Statist.}, 38(3):1793--1844, 2010.

\bibitem{sznitman}
A.S. Sznitman.
\newblock Topics in propagation of chaos.
\newblock In {\em Ecole d'Ete de Probabilit\'{e}s de Saint-Flour XIX}, volume
  1464 of {\em Lecture Notes in Mathematics}, pages 165--251, Berlin, 1991.
  Springer.

\bibitem{talayvaillant}
D.~Talay and O.~Vaillant.
\newblock A stochastic particle method with random weights for the computation
  of statistical solutions of {M}c{K}ean-{V}lasov equations.
\newblock {\em The Annals of Applied Probability}, 13(1):140--180, 2003.

\bibitem{tarrago}
P.~Tarrago.
\newblock Spectral deconvolution of unitary invariant matrix models.
\newblock arXiv:2006.09356, 2020.

\bibitem{transolstat}
V.C. Tran.
\newblock A wavelet particle approximation for {M}c{K}ean-{V}lasov and
  {N}avier-{S}tokes spatial statistical solutions.
\newblock {\em Stochastic Processes and their Applications}, 118(2):284--318,
  2008.

\bibitem{TsyCRAS}
Alexandre Tsybakov.
\newblock On the best rate of adaptive estimation in some inverse problems.
\newblock {\em C. R. Acad. Sci. Paris S\'{e}r. I Math.}, 330(9):835--840, 2000.

\bibitem{tzvetkov}
N.~Tzvetkov.
\newblock Random data wave equations.
\newblock \verb"arXiv:1704.01191", 2017.

\bibitem{vergassollaetal}
M.~Vergassola, B.~Dubrulle, U.~Frisch, and A.~Noullez.
\newblock Burgers'equation, {D}evil's staircases and the mass distribution for
  large-scale structures.
\newblock {\em Astronomy and Astrophysics}, 289:325--356, 1994.

\bibitem{vishikfursikov}
M.J. Vishik and A.V. Fursikov.
\newblock {\em Mathematical Problems of Statistical Hydromechanics}.
\newblock Mathematics and its Applications. Kluwer Academic Publishers, 1980.

\bibitem{voiculescu}
D.~Voiculescu.
\newblock Addition of certain noncommuting random variables.
\newblock {\em J. Funct. Anal.}, 66(3):323--346, 1986.

\bibitem{voiculescu93}
D.~Voiculescu.
\newblock The analogues of entropy and {F}isher's information measure in free
  probability theory, i.
\newblock {\em Commun. Math. Phys.}, 155:71--92, 1993.

\end{thebibliography}
\end{document}